\documentclass[10pt, final, journal, letterpaper, oneside, twocolumn]{IEEEtran}

\makeatletter

\let\proof\@undefined
\let\endproof\@undefined

\let\theorem\@undefined
\let\endtheorem\@undefined
\makeatother
\usepackage{braket}
\let\proof\relax 
\let\endproof\relax
\usepackage{amsthm}
\usepackage{thmtools}
\usepackage{mathtools}
\usepackage{bigints} 	
\usepackage{steinmetz}  
\usepackage{amsmath,amssymb,amsfonts,bm} 
\usepackage{graphics}  
\usepackage{graphicx}
\usepackage{cite}

\usepackage{enumitem}
\usepackage{multirow}
\usepackage{array}
\usepackage{setspace}
\usepackage{upgreek}
\usepackage{fancyhdr}
\usepackage{float}
\usepackage{tocloft}
\usepackage{tcolorbox}
\usepackage{comment}
\usepackage{empheq, blkarray}
\usepackage{subfig}
\usepackage{subfloat}
\makeatletter
\newcommand{\removelatexerror}{\let\@latex@error\@gobble}
\makeatother
\usepackage{verbatim}
\tcbuselibrary{breakable}
\newcommand\numberthis{\addtocounter{equation}{1}\tag{\theequation}}

\theoremstyle{definition}
\newtheorem{definition}{Definition}
\newtheorem{assumption}{Assumption}
\theoremstyle{plain}
\declaretheorem[name=Theorem]{thm}

\declaretheorem[name=Proposition]{prop}

\theoremstyle{remark}
\newtheorem{remark}{Remark}
\newtheorem{example}{Example}

\DeclareMathOperator{\dom}{dom}

\DeclareMathOperator{\interior}{int}
\DeclareMathOperator{\lse}{lse}
\DeclareMathOperator{\argmax}{argmax}

\DeclareMathOperator{\A}{\mathbf{A}}
\DeclareMathOperator{\B}{\mathbf{B}}
\DeclareMathOperator{\C}{\mathbf{C}}
\DeclareMathOperator{\D}{\mathbf{D}}

\let\P\relax
\DeclareMathOperator{\P}{\mathbf{P}}

\usepackage{caption}

\usepackage[colorlinks=true,
linkcolor=black,
urlcolor=black,
citecolor=black]{hyperref}
\usepackage{cleveref}

\begin{document}
\title{\LARGE On Passivity, Reinforcement Learning and Higher-Order Learning in Multi-Agent Finite Games} 
\author{Bolin Gao and Lacra Pavel%
\thanks{This work was supported by NSERC Grant (261764). 
	B. Gao and L. Pavel are with the Department of Electrical and Computer Engineering, University of Toronto, 
	Canada. 
	{\tt\small bolin.gao@mail.utoronto.ca, pavel@ece.utoronto.ca}}%
}

\maketitle

\IEEEpeerreviewmaketitle
\begin{abstract}
	In this paper, we propose a passivity-based methodology for analysis and design of reinforcement learning in multi-agent finite games. Starting from a known exponentially-discounted reinforcement learning scheme,  we show that convergence to a Nash distribution can be shown in the class of games characterized by the monotonicity property of their (negative) payoff. We further exploit passivity to propose  a class of higher-order schemes that  preserve convergence properties, can improve the speed of convergence and can even converge in cases whereby their first-order counterpart fail to converge. We demonstrate these properties through numerical simulations for several representative games.  
\end{abstract}

\section{Introduction}

In multi-agent reinforcement learning in games, an agent repeatedly adjusts his strategy in response to a stream of payoffs which are in turn dependent on the actions of the other agents. The objective  is for agents to  arrive at a strategy profile that yields the best obtainable outcome for each of them, under information-constraints or in the presence of noise. Whether such an action profile can be reached depends on the learning process used by the collection of these agents. 
  Learning in finite games typically involves  discrete-time stochastic processes, where stochasticity arises  from the agents' randomized choices, \cite{Erev,Fudenberg,Leslie,Chasparis_and_Shamma,Coucheney,Cominetti,Laraki}. A key approach to analyzing such processes is based on the ordinary differential equation (ODE) method of stochastic approximation, a method which relates  their  behaviour to that of a ``mean field" ODE, \cite{Benaim1999}. Motivated by  this, we follow  \cite{Coucheney}, \cite{Merti_Learning,ShammaTAC2005,Sorin2009}  and consider a reinforcement learning scheme directly in continuous-time. By so doing, we are  able to focus  on the relationship between reinforcement learning, convex analysis, and passivity. 
\footnote{One could use the results developed in this paper to analyze discrete-time learning schemes as in  \cite{Leslie}, \cite{Coucheney},  but we leave this for future work.} 

Our starting point is a variant of the continuous-time exponentially-discounted learning (EXP-D-RL) in \cite{Coucheney}, versions of which are also  known as  the exponential-weight algorithm \cite{Sorin2009}, or Q-learning, \cite{Leslie}, \cite{Cominetti}. Under this continuous-time learning process, each agent maintains a vector of scores for all his actions, based on aggregation of his exponentially-discounted  stream of payoffs. These scores are then converted into mixed strategies using a static logit  (soft-max) rule which assigns choice probabilities proportionally to the exponential of each action's score. The resulting learning  dynamics describes the evolution of the scores (dual variables), rather than of the  mixed strategies (primal variables).  The same score dynamics results also from stochastic-approximation of a Q-learning scheme, which was shown  in \cite{Leslie} to converge to a Nash distribution (logit equilibria) in 2-player zero-sum  and partnership (potential) games. It is also related to the scheme proposed in \cite{Cominetti} for traffic games  in a repeated-game setup and  shown to converge to a Nash distribution in $N$-player potential games. The strategy dynamics induced by this score dynamics  is analyzed in \cite{Coucheney} and proved to converge  towards logit equilibria in potential games. With  the exception of \cite{Leslie}, the majority of these works  have focused exclusively on convergence  
in potential games,  \cite{Chasparis_and_Shamma,Coucheney,Cominetti}, \cite{Zino}. 
In contrast, less attention has been paid to stable games \cite{Hofbauer,Sandholm}, which encompass zero-sum games, potential games with concave payoffs and include well-known examples from evolutionary game theory such as the Rock-Paper-Scissors (RPS) game. 

Motivated by the above, in this paper, we  exploit passivity techniques and 
the natural monotonicity property associated with this class of games to show convergence to a Nash distribution. 
 The use of passivity to investigate game dynamics was first proposed in \cite{Fox_and_Shamma} for population games, based on the notion of $\delta$-passivity. The authors showed that certain game dynamics and the class of stable games, naturally satisfy this type of passivity.  The coupling between a $\delta$-passive system with a stable game implies a stable behaviour in the closed-loop solution. \cite{Mabrok} showed that if an evolutionary dynamic does not satisfy a
passivity property, then it is possible to construct a higher-order stable game that results in instability.   
Here 
we use an alternative concept, that of equilibrium-independent passivity, \cite{MA11}. This allows us to directly address convergence towards equilibria. We note that equilibrium-independent passivity was used recently in \cite{DianTAC} for  continuous-kernel games,  to relax the assumption on perfect-information on other players' actions.

\textit{Contributions.} 
Our contributions are twofold: we show (1) that a passivity framework can be used to prove convergence of reinforcement learning in finite games, and (2) that its principles can be used towards designing higher-order learning dynamics that preserve convergence to a Nash distribution.  Our approach is  based on  reformulating  the overall  learning dynamics of all agents as a feedback  interconnected system. We show that the continuous-time EXP-D-RL scheme, \cite{Coucheney},  can be naturally posed as a payoff-feedback  system in the dual space, 
 where the  forward system 
 satisfies more than passivity, namely it is output strictly (equilibrium-independent) passive. We exploit its particular storage function, related to a Bregman divergence, as a Lyapunov function and  
 show convergence of EXP-D-RL to a Nash distribution (logit equilibrium) in  any $N$-player game  for which the (negative) payoff is merely monotone, not necessarily strict.  To the best of our knowledge, this is the first convergence result for such games.  This class of games corresponds to the class of  stable games in population games \cite{Sandholm}, and to  the class of monotone games in continuous games, \cite{Facchinei}. It subsumes,  in the case of finite-action games,  potential games with concave potential, 2-player zero-sum games,  as in \cite{Leslie}, as well as the standard 2-player RPS game. 
 Key to our approach is the fact that we analyze convergence based on the natural, score dynamics, which are in the dual  (payoff) space, as in \cite{Cominetti}. This is contrast to analyzing the induced strategy dynamics as  done in \cite{Coucheney}, and is unlike the indirect analysis done in \cite{Leslie}, via connection to the smooth best-response. Unlike \cite{Coucheney,Cominetti}, we show convergence to a Nash distribution in games that go beyond the class of potential games,  for example the RPS games or the Shapley game. 
Furthermore, we exploit cocoercivity of the  soft-max  to show convergence even for hypo-monotone games (corresponding to  unstable games in \cite{Sandholm}), such as unstable RPS games and Shapley's game, is the temperature parameter is above a certain threshold. To achieve this, we balance the game shortage of passivity (hypo-monotonicity) by the excess passivity coming from the soft-max map (cocoercivity). 


In the second part of the paper, we build on the passivity interpretation and propose a method to design higher-order extensions of  EXP-D-RL. 
Higher order dynamics, via the introduction of auxiliary states,  can have different properties. They can have significant benefits fostering convergence in larger classes of games, as shown in \cite{ShammaTAC2005,ArslanShamma2006} for fictitious-play/gradient-play and in evolutionary games. In reinforcement learning, higher-order  extensions of the un-discounted reinforcement learning have been proposed in \cite{Laraki} based on second or  $n$-th order payoff-integration (equivalent to a cascade modification  of the first-order replicator dynamics by a chain of integrators). These dynamics have been shown to lead to the elimination of weakly dominated strategies, followed by the iterated deletion of strictly dominated strategies, a property not exhibited by standard replicator dynamics. However, as shown in \cite{Mabrok}, such cascade augmentation does not guarantee that passivity/convergence properties are preserved when extending from first to higher-order  dynamics. Our second result shows that if higher-order dynamics are built by feedback modification via a passive system that preserves the equilibrium point, convergence to a Nash distribution can be guaranteed in the same class of games. 
We explicitly build a second-order learning scheme  based on this method, by specifying a particular LTI positive-real system for the feedback modification path. 
We show numerically that these higher-order dynamics can converge faster and, in some cases, can converge in larger classes of games (more hypo-monotone) than the first-order scheme. 
A short version of this paper will appear in \cite{Bo_LP_CDC2018}.

The paper is organized as follows. Section II provides background material. Section III introduces the continuous-time score-based EXP-D-RL reinforcement learning scheme. Section IV provides convergence analysis of the first-order EXP-D-RL scheme.   
Section V proposes and analyzes a class of higher-order dynamics. Section VI discusses connections to population games. Section VII discusses several examples and presents simulation results.  
Section VIII provides the conclusions. 

\vspace{-0.2cm}
\section{Background}


\vspace{-0.2cm}
\subsection{Convex Optimization and Monotone Operator Theory}

The following is from \cite{Facchinei}, \cite{Boyd}. Let $z \in \mathbb{R}^n$,  $z\!=\! [z_1,...,z_n]^\top  $, also denoted as  $z\!=\! (z_1,...,z_n)$  or $z\!= \!(z_i)_{i =\{1,\dots,n\}}$. 
We assume that $\mathbb{R}^n$ is equipped with the standard inner product $\langle z, z^\prime \rangle \coloneqq \sum_{i =1}^n z_iz^\prime_i = z^\top   z^\prime$ and the induced 2-norm $\|z\|_2\coloneqq \sqrt{\langle z, z \rangle}$.  

An operator (or mapping) $F: \mathcal{D} \subseteq \mathbb{R}^n \to \mathbb{R}^n$ is said to be monotone on $\mathcal{D}$ if $(F(z) - F(z^\prime))^\top (z-z^\prime) \geq 0, \forall z,z^\prime \in \mathcal{D}$. It is strictly monotone if strict inequality holds $ \forall z,z^\prime \in \mathcal{D}, z \neq z^\prime$. 
$F$ is $\mu$-strongly monotone if $(F(z) - F(z^\prime))^\top (z-z^\prime) \geq \mu \| z - z^\prime \|_2^2, \forall z,z^\prime \in \mathcal{D}$, for some $\mu >0$. 
We note that a $C^1$ function $f$ is convex if and only if $(\nabla f(z) - \nabla f(z^\prime))^\top (z - z^\prime) \geq 0, \forall z, z^\prime \in \dom f$, where $\nabla f$ is its gradient,  and strictly convex if and only if $	(\nabla f(z) - \nabla f(z^\prime))^\top (z - z^\prime) > 0, \forall z, z^\prime \in \dom f, z \neq z^\prime$.  
$F: \mathcal{D} \subseteq \mathbb{R}^n \to \mathbb{R}^n$ is L-Lipschitz if there exists a $L > 0$ such that $\|F(z) - F(z^\prime)\|_2 \leq L \|z - z^\prime\|_2, \forall z, z^\prime \in \mathcal{D}$. $F$ is  nonexpansive  if $L=1$, and contractive if $L \in (0,1)$.  $F: \mathcal{D} \subseteq \mathbb{R}^n \to \mathbb{R}^n$ is  \textit{$\beta$-cocoercive} if there exists a $\beta > 0$ such that $(F(z) - F(z^\prime)^\top (z-z^\prime ) \geq \beta\|F(z) - F(z)^\prime\|_2^2, \forall z, z^\prime \in \mathcal{D}$. $F$ is referred to as \textit{firmly nonexpansive} for $\beta=1$.

\vspace{-0.2cm}
 \subsection{Equilibrium Independent Passivity}
The following  are obtained from \cite{MA11}, \cite{DN13}. Consider $\Sigma$ \vspace{-0.2cm}
\begin{align} \label{eq:DynOverall}
\Sigma :
  \begin{cases}
    \dot{z} = f(z,u), & \\
    y = h(z,u), & \\
  \end{cases}
\end{align}
with $z\in \mathbb{R}^n$,  $u \in \mathbb{R}^q$, $y \in \mathbb{R}^q$, $f$ locally Lipschitz, $h$ continuous. Consider a differentiable  function $V:\mathbb{R}^n \!\rightarrow \! \mathbb{R}$. The time derivative of $V$ along solutions of (\ref{eq:DynOverall}) is  $\dot{V}(z) = \nabla  V(z)^\top  f(z,u)$ or just $\dot{V}$. Let  $\overline{u}$, $\overline{z}$, $\overline{y}$ be  an equilibrium condition, such that $0\!=\!f(\overline{z},\overline{u})$, $\overline{y}\!=\!h(\overline{z},\overline{u})$.  Assume there exists $ \Gamma\subset\mathbb{R}^q$ and a continuous function $k_z(\overline{u})$  such that for any constant  $\overline{u} \in \Gamma$,   $f(k_z(\overline{u}),\overline{u}) \!=\! 0$ (basic assumption).  
\begin{definition}\label{def:EIP}
System $\Sigma$ (\ref{eq:DynOverall}) is Equilibrium Independent Passive (EIP) if it is passive with respect to $\overline{u}$ and $\overline{y}$; that is for every $ \overline{u} \in \Gamma$ there exists a differentiable, positive semi-definite storage function $V_{\overline{z}}: \mathbb{R}^n \to \mathbb{R}$ such that $V_{\overline{z}}(\overline{z}) = 0$ and, for all $u \in \mathbb{R}^q$, $z \in \mathbb{R}^n$,\vspace{-0.2cm}
		\begin{equation}\label{eq:EIP}
		\dot{V}_{\overline{z}}(z)  \leq  (y-\overline{y})^\top  (u-\overline{u}),
		\end{equation}
		$\!\Sigma$  is output-strictly EIP (OSEIP) if there is a $\!\beta>0$ such that 	\vspace{-0.2cm}
				\begin{equation}\label{eq:OSEIP}
		\dot{V}_{\overline{z}}(z)  \leq  (y-\overline{y})^\top  (u-\overline{u}) - \beta \|y-\overline{y}\|_2^2,
				\end{equation}
\end{definition}
EIP requires that  \eqref{eq:EIP} holds for every $ \overline{u} \in \Gamma$ ($\Sigma$ to be passive independent of the equilibrium point), while traditional passivity, \cite{Khalil} requires that it holds only for a particular $ \overline{u}$ (usually associated with the origin as equilibrium). 
EIP properties help in deriving stability and convergence properties for feedback systems without requiring exact knowledge of an equilibrium point, but rather  only that it exists. The parallel interconnection of  two EIP systems is an EIP system, and the feedback interconnection of two EIP systems that satisfies the basic assumption is an EIP system (cf. Property 2 and 3 in \cite{MA11}). 
  When system $\Sigma$ is just a static map, EIP   is equivalent to incrementally passivity. 
  and to monotonicity. A static nonlinear function $y=F(u)$ is defined to be EIP (OSEIP) 
  if $F$  is monotone ($\beta$-cocercive). 
A linear (output strictly) passive system, $\dot{z} = Az + Bu$, $y = Cz + Du$, with $(A, B)$  controllable, $(A, C)$ observable, and $A$ invertible is (OS)EIP (cf. Ex. 1 in \cite{MA11}). This can be shown using $V(z-\overline{z})$, where $V$ is the quadratic storage function associated with the passivity of the linear system relative to the origin equilibrium, by direct application of the KYP lemma (cf. Section 6.4 of \cite{Khalil}). The additional requirement of invertibility of $A$ is necessary to satisfy the basic assumption on the existence and continuity of $k_z$, which is defined by $k_z(u) = -A^{-1}Bu$. The equilibrium input-output map is defined by $k_y(u) = (- CA^{-1}B + D)u$.

\vspace{-0.2cm}
\subsection{Games in Normal Form}

Consider a game $\mathcal{G}$ between a set of players (agents) $\mathcal{N} =  \{1, \ldots, N\}$, where each player $p \in \mathcal{N}$ has a finite set of actions (or pure strategies) $\mathcal{A}^p $, and a payoff $\mathcal{U}^p: \mathcal{A} \to \mathbb{R}$, with  $\mathcal{A} = \prod_{p \in \mathcal{N}} \mathcal{A}^p$ the overall action set of all players, \cite{Merti_Learning}.  

 Let  $|  \mathcal{A}^p| = n^p$ and $n=\sum_{p \in \mathcal{N}} n^p$. Without loss of generality we identify $\mathcal{A}^p$ as the  corresponding index set, i.e., $ \mathcal{A}^p= \{1, \ldots, n^p\}$ and denote a generic action  as   $i \in \mathcal{A}^p$.  
Let $x^p = (x^p_i)_{i \in \mathcal{A}^p}$ denote the mixed strategy of player $p$, a probability distribution over his set of actions $\mathcal{A}^p$. Then $x^p \in \Updelta^p$, where 
$\Updelta^p \coloneqq \{x^p \in \mathbb{R}^{n^p}_{\geq 0}| \|x^p\|_1 = 1\}$ 
 is the set of mixed strategies for player $p$. A mixed strategy profile is denoted as $x = (x^1,\ldots, x^N) \in \Updelta$, where $\Updelta \coloneqq \prod_{p \in \mathcal{N}} \Updelta^p$ 
 is called the game's strategy space. 
  We also use the shorthand notation $x=(x^p; x^{-p})$  where $x^{-p}$ is the strategy profile of the other players except $p$. 
 Player $p$'s expected  payoff  to using $x^p$ in the mixed  strategy profile $x = (x^1,\ldots, x^N)\in \Updelta$ is 
    \vspace{-0.2cm}
\begin{equation}\label{eq_u_calU}
\mathcal{U}^p(x) = \sum\limits_{i_1 \in \mathcal{A}^1} \cdots \sum _{i_N \in \mathcal{A}^N} \mathcal{U}^p(i_1, \ldots, i_N) x^1_{i_1}  \ldots x^N_{i_N},
\end{equation}
where  
 $\mathcal{U}^p(\bm{i}) \equiv \mathcal{U}^p(i_1, \ldots, i_N)$ denotes his payoff in the pure (action) profile $\bm{i}=\{i_1,  \ldots, i_N\} \in \mathcal{A}$, with 
 $i_p \in \mathcal{A}^p$. 
We denote by  $U_i^p(x) \equiv \mathcal{U}^p(i; x^{-p}) \equiv \mathcal{U}^p(x^1, \ldots, i, \ldots, x^N)$,
his  expected payoff corresponding to using pure strategy $i \in \mathcal{A}^p$ in  the mixed profile $x \in \Updelta$. 
Note that we can write \eqref{eq_u_calU} as \vspace{-0.2cm}
\begin{equation}\label{eq_u_U}
\mathcal{U}^p(x) 
 =  \sum\limits_{i \in \mathcal{A}^p} \mathcal{U}^p(i; x^{-p}) \, x^p_i = \sum\limits_{i \in \mathcal{A}^p} x^p_i \,U_i^p(x),
\end{equation}
or $\mathcal{U}^p(x)  = {x^p}^\top   U^p(x)$, 
where  $U^p(x) = (U_i^p(x))_{i \in \mathcal{A}^p}  \in \mathbb{R}^{n^p}$ is called the payoff  vector of player $p$ at $x \in \Updelta$, 
 indicating the duality pairing between $x^p$ and $U^p$, \cite{Merti_Learning}. With the players' (expected) payoffs $\mathcal{U}^p : \Updelta \rightarrow \mathbb{R}$, the tuple $(\mathcal{N}, \Updelta, \mathcal{U})$ is called  the mixed extension of $\mathcal{G}$ also denoted by $\mathcal{G}$.



\begin{definition}
	\label{def:Nash_Equilibrium}
	 A mixed strategy profile $x^\star \in \Updelta$ is a Nash equilibrium of game $\mathcal{G}$ if \vspace{-0.2cm}
	\begin{equation}\label{eq:defNash_Equilibrium}
 \mathcal{U}^p(x^\star) \geq  \mathcal{U}^p(x^p; {x^\star}^{-p}), \, \forall  x^p \in \Updelta^p, \forall p \in \mathcal{N}.
	\end{equation}
\end{definition}
\noindent Define the \textit{mixed best-response map} of player $p$, $BR^p: \Updelta \rightarrow \Updelta^p$, $x \mapsto \textstyle{\argmax}_{x^p \in \Updelta^p} \, \mathcal{U}^p(x)$. 
Then $x^\star $ satisfies  \vspace{-0.2cm}
	\begin{equation}\label{eq:NE_BR}
	{x^\star}^p \in BR^p({x^\star}), \, \forall p \in \mathcal{N}, 
	\end{equation}
or  ${x^\star} \in BR({x^\star})$, where $BR = (BR^p)_{p \in \mathcal{N}}$, i.e., ${x^\star}$ is a fixed-point of the overall best-response map of all players. 
Alternatively, by  \eqref{eq_u_U}, \eqref{eq:defNash_Equilibrium}, ${{x^\star}^p}^\top  U^p(x^\star) \geq {{x}^p}^\top  U^p(x^\star),  \forall  x^p \in \Updelta^p, \forall p \in \mathcal{N},$  or, by concatenating them, 	\begin{equation}\label{eq:NE_VI}	-(x-x^\star)^\top  U(x^\star) \geq 0,  \forall  x \in \Updelta	\end{equation}where  $U: \Updelta \to \mathbb{R}^n$,  $U(x)= (U^p(x))_{p \in \mathcal{N}} $,  denotes  the payoff vector  of the entire set of players, or the \emph{game (static) map}. Thus  $x^\star$ can be equivalently characterized  as a solution of the variational inequality, $\textup{VI}(-U, \Updelta)$,  \eqref{eq:NE_VI},  \cite{Facchinei}.

\vspace{-0.2cm}
\section{Exponentially-Discounted Learning (EXP-D-RL)} 

 


In the following, we describe the score-based reinforcement learning (RL) scheme, modeled in continuous-time as in  \cite{Coucheney,Merti_Learning,Laraki}.  Each player keeps  a \textit{score} $z^p$  based on his received payoff  $U^p$, and maps it into a  strategy $x^p\in \Delta^p$.  
He plays the game $\mathcal{G}$ according to the strategy $x^p$. This process is repeated indefinitely, with an infinitesimal time-step between each stage; hence can be modeled in continuous-time, 
as a three stage process, described  as follows:\\
(1) \textbf{Assessment Stage:} Each player $p$ keeps a vector \textit{score variable} $z^p(t) \in \mathbb{R}^{n^p}$,  $z^p = (z^p_i)_{i\in \mathcal{A}^p}$, with each $i$-th action  $i \in  \mathcal{A}^p$ having a score $z_i^p \in \mathbb{R}$. Starting from an initial score, he updates it based on  \textit{exponentially-discounted aggregation (learning)} (EXP-D-RL) of his own payoff stream, \vspace{-0.3cm} 
\begin{equation}
z_i^p(t) =e^{-\gamma t }z^p_i(0) +  \gamma \int\limits_{0}^t  e^{-\gamma (t - \tau)} u_i^p(\tau) \mathrm{d}\tau,
\label{def:exponential_discount_learning}\tag{EXP-D-RL}
\end{equation} 
or, in differential form,  \vspace{-0.2cm}
\begin{equation}
\dot z^p_i = \gamma(u^p_i - z^p_i), \quad z^p_i(0) \in \mathbb{R}, \, \, \, \forall i \in  \mathcal{A}^p.
\label{eqn:zp_i}
\end{equation}
where $u^p_i(t) = U^p_i(x(t))$,  $\gamma>0$ is the learning rate and   $z^p_i(0)$ is the bias towards strategy $i$ at the beginning of the game.   
EXP-D-RL  can be represented by a  scalar-valued transfer function $G(s) = \dfrac{\gamma}{s+\gamma}$,  similar to  the scheme studied in \cite{Coucheney}, where $G(s) = \dfrac{1}{s+T}$. The case $G(s) = \dfrac{1}{s}$, i.e., integration of the (un-discounted) payoff is studied in  
\cite{Merti_Learning},\cite{Laraki}. 
 \\
(2) \textbf{Choice Stage}: Each player  $p$ maps his own score  $z^p \in \mathbb{R}^{n^p}$, into a mixed strategy $x^p \in \Updelta^p$ using a \textit{choice map}, e.g. 
the \textit{best-response choice}, $M^p: z^p \mapsto \textstyle{\argmax}_{x^p \in \Updelta^p} \thinspace {x^p}^\top  z^p$. 
To ensure that  $M^p$ is single-valued, an at-least strictly convex function $\psi^p$ called \textit{regularizer} is used, \cite{Merti_Learning}, which yields the \textit{regularized/smooth best-response choice}, \vspace{-0.2cm}
\begin{equation} 
	\sigma^p: z^p \mapsto \textstyle{\argmax}_{x^p \in \Updelta^p} \thinspace \{{x^p}^\top  z^p - \psi^p(x^p)\}. 
	\label{def:regularized_br_map}
\end{equation}
Depending on the context, the regularizer is also referred to as admissible deterministic perturbation, penalty function, smoothing function, barrier function or Bregman function.  For detailed construction of the regularizer, \cite{Merti_Learning}. 
We consider the commonly used (negative) Gibbs entropy,\vspace{-0.2cm}
\begin{equation} 
	\psi^p(x^p) \coloneqq \epsilon \textstyle{\sum}_{j \in \mathcal{A}^p} x^p_j \log(x^p_j), \epsilon > 0,
	\label{def:neg_entropy}
\end{equation}
for which 
the choice map \eqref{def:regularized_br_map}  is the \textit{soft-max function} \cite{Merti_Learning},\vspace{-0.2cm}
\begin{equation}
	\sigma^p(z^p) \! \coloneqq  \!\frac{1}{\sum_{j \in \mathcal{A}^p} \!\exp(\frac{1}{\epsilon}z^p_j)} \!\begin{bmatrix}\exp(\frac{1}{\epsilon} z^p_1)\! \ldots\! \exp(\frac{1}{\epsilon} z^p_n) \!\end{bmatrix}^\top 
	\label{def:soft-max}
\end{equation}where $\epsilon > 0$. Note that $\sigma^p: \mathbb{R}^{n^p} \rightarrow \interior(\Updelta^p)$,  where  $\interior(\Updelta^p) =  \{x^p \in \mathbb{R}^{|\mathcal{A}^p|}| \|x^p\|_1 = 1, x^p_i > 0\}$ is the (relative) interior of the simplex $\Updelta^p$. 
$\epsilon $ is typically referred to as the \textit{temperature parameter}. For $\epsilon = 1$, \eqref{def:soft-max} is known as the \textit{standard soft-max function}. 
As $\epsilon \to \infty$, actions are selected with uniform probability, and as $\epsilon \to 0$, the soft-max function selects the action associated with the highest score, provided that the difference between any two scores is not too small. 
With $\sigma^p$, the mixed strategy for player $p$ is taken as,\vspace{-0.2cm}
\begin{equation} \label{eqn:mapping_score_strategy_population} x^p(t) = \sigma^p(z^p(t)) := \begin{bmatrix} \sigma^p_1(z^p),  \ldots,\sigma^p_{n^p}(z^p) \end{bmatrix}^\top , \end{equation} where $x^p_i = \sigma^p_i(z^p), \forall i \in \mathcal{A}^p$. 

\noindent 
(3) \textbf{Game Stage}: Each player $p$ plays the game $\mathcal{G}$ according to his strategy $x^p$  and records his own payoff  vector $u^p \coloneqq U^p(x) = (U_i^p(x))_{i \in \mathcal{A}^p} $, where $u_i^p (t) = U_i^p(x(t))$, $\forall i \in \mathcal{A}^p$.   
Then, with $z^p = (z^p_i)_{i \in \mathcal{A}^p}$,  \eqref{eqn:zp_i}, $x^p$, \eqref{eqn:mapping_score_strategy_population}, and $\sigma^p$ in \eqref{def:soft-max}, the EXP-D-RL scheme for player $p \in \mathcal{N}$ is written as, \vspace{-0.2cm}
\begin{equation}\label{eqn:first_order_score_dyn} 
\begin{array}{ll}
\dot z^p &= \gamma( U^p(x) - z^p),  \quad z^p(0) \in \mathbb{R}^{n^p}\\ 
x^p &= \sigma^p(z^p), 
\end{array}
\end{equation}
\vspace{-0.4cm}\begin{remark} 
Similar forms of ``exponentially-discounted" score dynamics have been investigated in \cite{Coucheney}\cite{Cominetti},\cite{Kianercy}. 
When $\gamma = 1$, \eqref{eqn:first_order_score_dyn} coincides with the learning rule studied in \cite{Coucheney} for the \textit{discount factor} $T= 1$. When $\gamma \neq1$ there is a slight difference: we use $\gamma$ not only as a discount factor (as $T$ in \cite{Coucheney}), but also as a  multiplicative factor (learning rate), and by doing so the rest points of \eqref{eqn:first_order_score_dyn}  are independent of $\gamma$. 
Structurally, EXP-D-RL is similar to online mirror descent (OMD) in convex optimization, recently studied for continuous games in  \cite{Zhou},\cite{MertikCDC2017}. In particular, the score $z^p$ is the \textit{dual} variable to the \textit{primal} variable $x^p$. Therefore, \eqref{eqn:first_order_score_dyn} describes the evolution of learning in  the dual space ${\mathbb{R}^{n^p}}$, whereas the induced strategy trajectory describes the evolution in the primal space $\Updelta^p$. This type of duality is discussed in \cite{Merti_Learning}. 
Lastly, we note that using stochastic approximation, \eqref{eqn:first_order_score_dyn} corresponds to  the individual Q-learning algorithm, which has been shown to converge in 2-player zero-sum games and 2-player partnership games (Proposition 4.2, \cite{Leslie}). 
\end{remark}
\begin{remark}Note that a first-order Euler discretization of the dynamics  \eqref{eqn:first_order_score_dyn}, with discretization step $ \alpha$,  is \vspace{-0.2cm} \begin{align*}Z^p(k+1) &= Z^p(k) +  \alpha \gamma \Big (U^p(X(k))- Z^p(k)\Big)\\ X^p(k+1) &= \sigma^p(Z^p(k+1))\end{align*}where $X^p(k) = (X^p_i(k))_{i \mathcal{A}^p}$ is the mixed-strategy 
and $X^p_i(k)$ the probability of playing $i \in \mathcal{A}^p$ at the $k$-th instance of play. This recursion  tracks \eqref{eqn:first_order_score_dyn} arbitrarily well over finite-time horizons when the discretization step $\alpha$  is sufficiently small, but requires perfect monitoring of the mixed strategies of the other players. This can be relaxed if players are assumed to possess a bounded, unbiased estimate of their actions' payoffs,  or if they  observe their in-game  realized payoffs. In fact,  \eqref{eqn:first_order_score_dyn} is  the mean field of the discrete-time stochastic process, \cite{Benaim1999},\vspace{-0.2cm} \begin{align*}Z^p(k+1) &= Z^p(k)+ \alpha(k) \gamma \Big ( \hat{u}^p(k) - Z^p(k) \Big )\\ X^p(k+1) &= \sigma^p(Z^p(k+1))\end{align*}where   $\hat{u}^p(k)= (\hat{u}^p_i(k))_{i \in \mathcal{A}^p}$ is an unbiased estimator of  $U^p(X(k))$, i.e., such that $\mathbb{E}(\hat{u}^p(k)) = U^p(X(k))$, and $\{\alpha(k)\}$ is a diminishing sequence of step-sizes, e.g. $\frac{1}{k+1}$.  If player $p$ can observe the action profile $\bm{i}^{-p}(k)$ played by his opponents (or can otherwise calculate his strategies' payoffs), such an estimate is provided by   $\hat{u}^p_i(k) =\mathcal{U}^p(i;\bm{i}^{-p}(k))$. If instead  player $p$ can only observe the payoff  $\pi^p(k):=\mathcal{U}^p(i(k);\bm{i}^{-p}(k))$ of his chosen action, then a typical choice for $\hat{u}^p_i(k)$ is $\hat{u}^p_i(k) =\pi^p(k) /X^p_i(k)$, if $i(k)=i$, \cite{Leslie,Coucheney, Merti_Learning}, where division by $X^p_i(k)$ compensates for the infrequency with which the score of the $i$-th strategy is updated. Results from the theory of stochastic approximation \cite{Benaim1999} can be used to tie convergence of such discrete-time algorithms to the asymptotic behaviour of  \eqref{eqn:first_order_score_dyn}, (see \cite{Coucheney}). In this paper we restrict our focus to the continuous-time learning scheme, as in \cite{Merti_Learning,ShammaTAC2005,Sorin2009}.   \end{remark}
\vspace{-0.2cm} 
\section{Convergence of EXP-D-RL Dynamics} 
In this section we analyze the convergence of EXP-D-RL,\! \eqref{eqn:first_order_score_dyn}.  Let  $z \! = \!(z^p)_{p \in \mathcal{N}}  \!\in \!\mathbb{R}^n$, $x \!= \!(x^p)_{p \in \mathcal{N}}$ and $U(x) \!= \! (U^p(x))_{p \in \mathcal{N}}$ denote the score vector, stacked mixed-strategies and  the overall payoff vector for all players, respectively. 
With \eqref{eqn:first_order_score_dyn},  EXP-D-RL is written  all players as\vspace{-0.2cm}
 \begin{equation} 
	\begin{cases}
	\dot z = \gamma(U(x) - z), \quad z(0) \in \mathbb{R}^n\\
	{x} = \bm{\sigma}(z),\\
	\end{cases}
	\label{eqn:first_order_exponentially_discounted_dynamics_system}
\end{equation}
where $\bm{\sigma}(z) := (\sigma^p(z^p))_{p \in \mathcal{N}}$, $\bm{\sigma}: \mathbb{R}^{n} \rightarrow \interior(\Updelta)$. 


Given that the payoff functions $U^p(x)$ are Lipschitz and bounded on $\Updelta$, the scores $z^p(t)$ will remain finite for all $t\geq0$, so $x(t) = \bm{\sigma}(z(t))$ will be defined for all $t\geq0$ and $x(t) \in \Updelta$  for all $t\geq0$. Moreover, $x^p_i(t) >0$, $t\geq0$, hence  only strategy trajectories 
on the interior of the simplex are obtained. 
  \begin{figure}[ht]
\vspace{-0.2cm}	\begin{center}
		\includegraphics[scale=0.6]{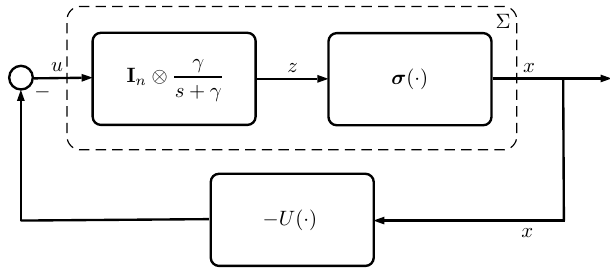}
	\end{center}
	\vspace{-0.3cm}
	\caption{EXP-D-RL  \eqref{eqn:first_order_exponentially_discounted_dynamics_system} as feedback interconnection} 
	\label{fig:Z_Exp}
	\vspace{-0.3cm}
\end{figure}
Note that with  ${u} = U(x)$, EXP-D-RL \eqref{eqn:first_order_exponentially_discounted_dynamics_system} can be represented as the  feedback interconnection in \autoref{fig:Z_Exp}, 
between $\Sigma$ 
\vspace{-0.2cm}
\begin{equation} \label{eq:sigma_u}
	\Sigma: \begin{cases}
	\dot z = \gamma(u - z)\\
	{x} = \bm{\sigma}(z),\\
	\end{cases}
\end{equation}
and ${u} = U(x)$,  where $U$ is the  game (static) map.  In the following we characterize the asymptotic behaviour of solutions of  \eqref{eqn:first_order_exponentially_discounted_dynamics_system}. 

\vspace{-0.2cm}
\subsection{Equilibrium Points of EXP-D-RL}
System \eqref{eqn:first_order_exponentially_discounted_dynamics_system} is written as,\vspace{-0.2cm}
\begin{equation}
\dot z = \gamma((U\circ\bm{\sigma})(z)  - z), \quad {x} = \bm{\sigma}(z),
\label{Exponential_Discount_Dynamic_2_Overall}
\end{equation}
where $U \circ \bm{\sigma}: \mathbb{R}^n \to \mathbb{R}^{n},  z \mapsto U(\bm{\sigma}(z))$. 
Note that $\bm{\sigma}:\mathbb{R}^n\rightarrow \Updelta$ is Lipschitz (we show in Proposition \ref{prop:prop_all_soft-max}) and since $U$ is Lipschitz and bounded, existence and uniqueness of global solutions  of \eqref{Exponential_Discount_Dynamic_2_Overall} follows from standard arguments. 
An equilibrium (rest) point $\overline{z}^\star$ 
is $\overline{z}^\star= (U\circ\bm{\sigma})(\overline{z}^\star)$, hence  a fixed point of the map $z \mapsto (U\circ\bm{\sigma})(z)$.  The existence of a fixed point is guaranteed by Brouwer's fixed-point theorem provided that $U \circ \sigma$ is a continuous function with bounded range \cite{Cominetti}. Based on \eqref{eq:sigma_u},  such a $\overline{z}^\star$ 
 can be represented as,\vspace{-0.2cm}
\begin{equation} 
\begin{cases}
\overline{u}^\star = \overline{z}^\star =U(\overline{x}^\star)\\
{\overline{x}^\star} = \bm{\sigma}(\overline{z}^\star).\\
\end{cases}
\label{eqn:fixed_point_condition_score_dyn}
\end{equation}
From \eqref{eqn:fixed_point_condition_score_dyn}, it follows that,  over the set of rest points 
of \eqref{Exponential_Discount_Dynamic_2_Overall},  the function $\overline{z}^\star \mapsto \bm{\sigma}(\overline{z}^\star)$ has an inverse given by  $\overline{x}^\star \mapsto U(\overline{x}^\star) $. 
As shown next, any $\overline{x}^\star =\bm{\sigma}(\overline{z}^\star)$ corresponding to \eqref{eqn:fixed_point_condition_score_dyn}  is a Nash equilibrium of a game with perturbed payoff (Nash distribution), and for small $\epsilon$ it approximates the Nash equilibria of $\mathcal{G}$,  (see also \cite{Govindan} and Proposition 2, \cite{Cominetti}). 
Note that as $\epsilon \rightarrow \infty$, all actions of each player $p$ are selected with uniform probability $ 1/n_p$ and there exists a unique $\overline{x}^\star$ at the centroid of the simplex $\Updelta$. 


\begin{prop}   
	\label{prop:perturbed_Nash_equilibria} Any $\overline{x}^\star =\bm{\sigma}(\overline{z}^\star)$, where $\overline{z}^\star$ is a rest point of \eqref{Exponential_Discount_Dynamic_2_Overall},  
	is a Nash equilibrium of  game $\mathcal{G}$ with perturbed payoff,\vspace{-0.2cm}
	\begin{equation}
	\label{eqn:perturbed_payoff_orig}
	\overline{\mathcal{U}}^p(x) = \mathcal{U}^p(x) - \epsilon \, {x^p}^\top   \mathbf{log}^p(x^p),
	\end{equation}
	where $\mathbf{log}^p(x^p) = \begin{bmatrix} \log(x^p_1), \ldots, \log(x^p_n) \end{bmatrix}^\top $, $\epsilon >0$.
\end{prop}
\begin{proof} Recall from \eqref{def:regularized_br_map}-\eqref{def:soft-max}, 
that for all $ p \in \mathcal{N}$, 
$	\sigma^p(z^p) = \textstyle{\argmax}_{x^p \in \Updelta^p} \thinspace \{{x^p}^\top  z^p - \epsilon\textstyle{\sum}_{j \in \mathcal{A}^p} x^p_j \log(x^p_j)\} =
	\textstyle{\argmax}_{x^p \in \Updelta^p} \thinspace {x^p}^\top  \big (z^p - \epsilon \mathbf{log}^p(x^p)\big ).
$	
 On the other hand, by \eqref{eq:NE_BR}  and \eqref{eq_u_U}, 
 a Nash equilibrium of $\mathcal{G}$ with perturbed payoff \eqref{eqn:perturbed_payoff_orig}, is a fixed-point of the  (perturbed) best-response function 
$\overline{BR}=(\overline{BR}^p)_{p \in \mathcal{N}}$, where 
	$\!	\overline{BR}^p (x)
	= \textstyle{\argmax}_{x^p \in \Updelta^p} {x^p}^\top   ( U^p(x) - \!\!\epsilon   \mathbf{log}^p(x^p))\!$ 
and  $U^p(x) $ is independent of $x^p$ (by \eqref{eq_u_U}).   
Thus, $\forall p \in \mathcal{N}$,  we can write $\overline{BR}^p (x) = \sigma^p(z^p)$, where $z^p=U^p(x)$, or,  concatenating all relations, $\overline{BR} (x) = \bm{\sigma}(z)$, where $z=U(x)$. A fixed-point of $\overline{BR}$ satisfies $x=\overline{BR} (x) = \bm{\sigma}(z)$, where $z=U(x)$, which are exactly \eqref{eqn:fixed_point_condition_score_dyn}. 
\end{proof}

\begin{remark}
	We note that  $\overline{x}^\star = \bm{\sigma}(\overline{z}^\star)$ corresponding to \eqref{eqn:fixed_point_condition_score_dyn} 
	is referred to as  a \textit{Nash distribution}, \cite{Leslie}, a \textit{(perturbed) logit equilibrium} of the game \cite[p. 191]{Sandholm}, or a type of quantal response equilibrium (QRE), \cite{McKelveyPalfrey1995}. 
	It satisfies $\overline{x}^\star = \bm{\sigma}(U(\overline{x}^\star)):=\overline{BR}(\overline{x}^\star))$, hence is a fixed-point of $\bm{\sigma} \circ U$,  
	which also exists by Brouwer's fixed point theorem (Theorem 1, \cite{McKelveyPalfrey1995}) and  is parameterized by $\epsilon$. 
At a Nash distribution $\overline{x}^\star$ each player plays a smooth best-response ($\sigma^p$) to payoffs arising from the others' play. 

\end{remark}

\vspace{-0.2cm}
\subsection{Soft-max Choice Map and Passivity of $\Sigma$} 
Recall the feedback representation of EXP-D-RL \eqref{eqn:first_order_exponentially_discounted_dynamics_system}  in \autoref{fig:Z_Exp}. The following results 
give properties for $\bm{\sigma} = (\sigma^p)_{p \in \mathcal{N}}$ and $\Sigma$, \eqref{eq:sigma_u}, related to passivity. First for $\sigma^p$, \eqref{def:soft-max}, $p \in \mathcal{N}$, consider 
\vspace{-0.2cm}
\begin{equation}
\lse^p(z^p) \coloneqq  \epsilon \ln\left(\textstyle{\sum}_{j \in \mathcal{A}^p} \exp(\epsilon^{-1} z^p_j)\right), \epsilon > 0, 
\label{def:log_sum_exp}
\end{equation}
the \textit{log-sum-exp function}, ($C^2$ and convex by \cite[p. 72]{Boyd}).

Since the following properties are valid for $\sigma^p$,  \eqref{def:soft-max}, for all  $p \in \mathcal{N}$,   without loss of generality, we consider $p = 1$ (drop the superscript) and denote  $\sigma$ and $\lse$. 
Note that  $\sigma$ is not injective and $\sigma(z+c\mathbf{1}) = \sigma(z)$, for all $z\in \mathbb{R}^n$,  $c \in \mathbb{R}$, where  $\mathbf{1} = \begin{bmatrix} 1 , \ldots , 1 \end{bmatrix}^\top  \in \mathbb{R}^n$.  

\begin{prop}\label{prop:prop_all_soft-max}
The soft-max function $\sigma: \mathbb{R}^n \to \interior(\Updelta)$ satisfies the following properties:\\
(i) $\sigma$  is gradient of  $\lse: \mathbb{R}^n \to \mathbb{R}$, that is, $\nabla \lse(z) = \sigma(z)$.\\
(ii) $\sigma$ is monotone, i.e.,\vspace{-0.2cm}
$$	(\sigma(z) - \sigma(z^\prime))^\top (z - z^\prime) \geq 0, \forall z, z^\prime \in \mathbb{R}^n,
$$
(iii) $\sigma$ is $\epsilon^{-1}$-Lipschitz, that is, \vspace{-0.2cm}
$$
	\|\sigma(z) - \sigma(z^\prime)\|_2 \leq  \epsilon^{-1}\|z - z^\prime\|_2, \forall z, z^\prime \in \mathbb{R}^n,
$$
where $\epsilon> 0$ is the temperature constant.\\
(iv) $\sigma$ is  $\epsilon$-cocoercive, that is, \vspace{-0.2cm}
$$
	(\sigma(z) - \sigma(z^\prime))^\top (z-z^\prime) \geq \epsilon \|\sigma(z) - \sigma(z^\prime)\|_2^2, \forall z, z^\prime \in \mathbb{R}^n.
$$
\end{prop}
\begin{proof} (i) Evaluating the partial derivative of $\lse$ in \eqref{def:log_sum_exp} at each component of $z$ yields $\dfrac{\partial \lse(z)}{\partial z_i} = \dfrac{\exp(\epsilon^{-1} z_i)}{\textstyle{\sum}_{j = 1}^n \exp(\epsilon^{-1} z_j)}= \sigma_i(z)$, for all $i = \{1,\dots,n\}$, hence $\nabla \lse(z) = \sigma(z)$.\\
(ii) The log-sum-exp function is  $C^2$ and convex, based on positive semi-definiteness of $\nabla^2\lse(z)$ \cite[p. 74]{Boyd}. Together with (i), this implies  that  $\sigma$ is monotone.\\
(iii) For the Hessian of $\lse$, it can be shown (\cite[p. 74]{Boyd}) that \vspace{-0.2cm}
$$	 v^\top \nabla^2\lse(z)v =  \epsilon^{-1} \Big (\textstyle{\sum}_{i = 1}^n v_i^2\sigma_i(z) - \left(\textstyle{\sum}_{i= 1}^n v_i\sigma_i(z)\right)^2 \Big ),$$
for all $z, v \!\in \! \mathbb{R}^n$, hence, $v^\top \nabla^2\lse(z)v \leq  
  \epsilon^{-1}\textstyle{\sup_z}\{\sigma_i(z)\}\sum\limits_{i = 1}^n v_i^2.$ 
Since $\sup_z\{\sigma_i(z)\} = 1, \forall i$ and $\nabla^2\lse(z)$ is positive semidefinite, this implies that \vspace{-0.2cm}
$$0\leq v^\top \nabla^2\lse(z)v \leq  \epsilon^{-1} \|v\|^2_2.$$
	By (i) and Theorem 2.1.6 in \cite{Nesterov},  $\sigma$ is $\epsilon^{-1}$-Lipschitz.\\	
(iv) 	$\epsilon$-cocoercivity of $\sigma$ follows directly from the Baillon-Haddad theorem \cite[p. 40]{Peypouquet}, since $\sigma$ is the gradient of $\lse$ by (i), and is $\epsilon^{-1}$-Lipschitz by (iii). 
 \end{proof} 
\vspace{-0.2cm}
\begin{remark}
	By  Proposition \ref{prop:prop_all_soft-max}(iv),\ $\sigma^p$ is $\epsilon$-cocoercive for all ${p \in \mathcal{N}}$, therefore, $\bm{\sigma} =(\sigma^p)_{p \in \mathcal{N}}$ is $\epsilon$-cocoercive,\vspace{-0.2cm}
	\begin{equation}\label{bm_Sigma_co}
		(\bm{\sigma}(z) - \bm{\sigma}(z^\prime))^\top (z - z^\prime) \geq \epsilon \| \bm{\sigma}(z) - \bm{\sigma}(z^\prime) \|_2^2, \forall z, z^\prime,
	\end{equation}
or, equivalently, output-strictly EIP (OSEIP). 
\end{remark}

Next, we characterize  $\Sigma$  \eqref{eq:sigma_u}. \vspace{-0.2cm}

\begin{prop}
	\label{prop:Sigma_OSEIP} 
	System $\Sigma$ \eqref{eq:sigma_u}  is output-strictly EIP (OSEIP). 	
\end{prop}
\begin{proof}
At an equilibrium condition of   \eqref{eq:sigma_u}, $\overline{z} = \overline{u}$ and $\overline{x} = \bm{\sigma}(\overline{z})$. 
Consider the following candidate storage  function,\vspace{-0.2cm}
	\begin{equation}\label{eq:V_storage}
	V_{\overline{z}}(z)  \!\!= \textstyle{\sum}_{p \in \mathcal{N}} \! \Big (\! \lse^p(z^p) - \lse^p({\overline{z}^p}) - \nabla\lse^p({\overline{z}^p})^\top (z^p-{\overline{z}^p}) \! \Big),
	\end{equation}
(the Bregman divergence of $\lse$, \cite{Nesterov}), with $V_{\overline{z}}(\overline{z}) =0$, $\lse^p$ as in \eqref{def:log_sum_exp}, convex. By Proposition \ref{prop:prop_all_soft-max}(i), $\nabla\lse^p = \sigma^p$. By convexity of $\lse^p$, $V_{\overline{z}}(z) \geq 0, \forall z \in \mathbb{R}^n$.  Taking the time derivative of $V_{\overline{z}}(z)$ along \eqref{eq:sigma_u} using  $\overline{z} = \overline{u}$ yields,\vspace{-0.2cm}
	\begin{align*}
	\dot V_{\overline{z}}(z) = & \nabla V_{\overline{z}}(z)^\top \dot z =  \textstyle{\sum}_{p \in \mathcal{N}} \gamma (\sigma^p(z^p) - \sigma^p({\overline{z}^p}))^\top (-z^p+u^p) \\
	= &\textstyle{\sum}_{p \in \mathcal{N}}\gamma(\sigma^p(z^p) - \sigma^p({\overline{z}^p}))^\top (-z^p + {\overline{z}^p} -{\overline{z}^p} + u^p) \\
	= & \gamma(\bm{\sigma}(z) - \bm{\sigma}({\overline{z}}))^\top [-(z - {\overline{z}}) + (u - {\overline{u}})],
	\end{align*}
	Using \eqref{bm_Sigma_co} in the foregoing yields\vspace{-0.2cm}
	\begin{equation}\label{OSEIP_Sigma}
	\dot V_{\overline{z}}(z)  \leq   - \gamma\epsilon \|\bm{\sigma}(z) - \bm{\sigma}({\overline{z}})\|_2^2 + \gamma(\bm{\sigma}(z) - \bm{\sigma}({\overline{z}}))^\top (u - {\overline{u}}),
	\end{equation}
With $x = \bm{\sigma}(z)$, $\overline{x} = \bm{\sigma}(\overline{z})$,   $\Sigma$,  \eqref{eq:sigma_u}  is OSEIP,  cf. \eqref{eq:OSEIP}.  \end{proof}
 \begin{remark}Note that $V_{\overline{z}}(z) =  \textstyle{\sum}_{p \in \mathcal{N}} D_{\lse^p}(z^p,\overline{z}^p)$, where $D_{\lse^p}(z^p,\overline{z}^p)$  is the Bregman divergence of $\lse^p$, i.e. the difference between the log-sum-exp function and its linear approximation at ${\overline{z}}^p$, \cite{Nesterov}. Using Theorem 2.1.5, \cite{Nesterov}, $\frac{\epsilon}{2} \| \bm{\sigma}(z)-\bm{\sigma}({\overline{z}})\|^2_2 \leq  V_{\overline{z}}(z) \leq \frac{\epsilon^{-1}}{2} \| z- {\overline{z}}\|^2_2$, $\forall z, \overline{z} \in \mathbb{R}^n$.\end{remark}



\vspace{-0.2cm}
\subsection{Convergence Analysis}

Next, based on the representation in  \autoref{fig:Z_Exp} (feedback interconnection  between $\Sigma$, \eqref{eq:sigma_u} and the static game map $-U$ on the feedback path),  we analyze the asymptotic behaviour of 
EXP-D-RL \eqref{eqn:first_order_exponentially_discounted_dynamics_system} 
by leveraging passivity properties of $\Sigma$.

First, solutions $z(t)$ of \eqref{eqn:first_order_exponentially_discounted_dynamics_system} remain bounded, (see also proof of Lemma 3.2, \cite{Leslie}). 
To see this, note that  since $U_i^p$ is continuous,  and $x\in \Delta$, $U_i^p(x)$ is bounded, i.e., 
there is some $M>0$  such that $|U_i^p(x) |\leq M$ for any $x \in \Delta$, for all $p \in \mathcal{N}$, $i \in \mathcal{A}^p$.  
From the integral form of EXP-D-RL, 
we can write  for all $p \in \mathcal{N}$, $i \in \mathcal{A}^p$,
$| z^p_i(t) | \leq e^{-\gamma t } |z^p_i(0)|+    \gamma \int\limits_{0}^t  e^{-\gamma (t - \tau)} |U_i^p(x(\tau))| \mathrm{d}\tau
           \leq e^{-\gamma t } |z^p_i(0)|+  M  (1-  e^{-\gamma t }) 
$
hence $| z^p_i(t) | \leq \max\{ |z^p_i(0)|,M\}$,  $\forall t\geq 0$. Also,  $\Omega = \{ z \in \mathbb{R}^n | \|z\|_2 \leq \sqrt{n} M\}$ is a compact, positively invariant set.

We make the following assumption on the payoff $U$. 

\begin{assumption}
	\label{assump:payoff_monotonicity}
(i) The negative payoff, $-U(\cdot)$, is monotone, \vspace{-0.3cm}
	\begin{equation}
	\label{eqn:U_Monotone_0}
	-(x-x^\prime)^\top (U(x) - U(x^\prime)) \geq 0, \, \, \forall x, x^\prime \in \Updelta.
	\end{equation}
(ii) The negative payoff, $-U(\cdot)$, is $\mu$-hypo-monotone, i.e., \vspace{-0.2cm}
	\begin{equation}
	\label{eqn:U_mu_Monotone}
	-(x-x^\prime)^\top (U(x) - U(x^\prime)) \geq -\mu \|x- x^\prime\|_2^2,
	\end{equation}
	for some $\mu > 0$,  for all $x, x^\prime \in \Updelta$. 
\end{assumption} 

\begin{remark}
The monotonicity in Assumption \ref{assump:payoff_monotonicity}(i) ($\mu=0$) is equivalent $-U(\cdot)$ being EIP and incrementally passive. In population games, \autoref{assump:payoff_monotonicity}(i) corresponds  to $\mathcal{G}$ being a \textit{stable game} cf. \cite{Hofbauer}, while in games with continuous actions it corresponds to games with monotone pseudo-gradient map, \cite{Facchinei}. Assumption \ref{assump:payoff_monotonicity}(ii) (as ``hypo-monotone" in \cite{Brogliato_Hypo}) corresponds to an \textit{unstable game} and is 
equivalent to shortage of monotonicity (passivity) of $-U(\cdot)$ as described by $\mu >0$. 
\end{remark}

\begin{remark} As in stable games, \cite{Sandholm}, Assumption \ref{assump:payoff_monotonicity}(i)  can be characterized via $y^\top  DU(x) y \leq 0$, for all $x \in \Updelta$, $y \in T\Updelta$,  and Assumption \ref{assump:payoff_monotonicity}(ii) via $y^\top  DU(x) y \leq \mu \|y\|_2^2$,  for all $x \in \Updelta$, $y \in T\Updelta$, where $D U(x)$ is the Jacobian of $U(x)$ and $T \Updelta = \prod_{p \in \mathcal{N}} T \Updelta^p$, $ T \Updelta^p  \coloneqq \{y^p \in \mathbb{R}^{n^p}|  \sum_{i \in \mathcal{A}^p} y^p_i = 0\}$ is the tangent  space of $\Updelta^p$. 
For $2$-player games,  $U$ is linear and \eqref{eqn:U_Monotone_0} can be checked based on the payoff matrices of the two players, $\mathbb{A}$ and $\mathbb{B}$. 
Since  \vspace{-0.2cm}\[	U(x) = \begin{bmatrix} U^1(x) \\ U^2(x) \end{bmatrix}= \begin{bmatrix} 0 & \mathbb{A} \\ \mathbb{B}^\top  & 0 \end{bmatrix} \begin{bmatrix}  x^1 \\  x^2 \end{bmatrix} :=\Phi x, \]hence \eqref{eqn:U_Monotone_0} is  $-(x-x^\prime)^\top  \Phi (x-x^\prime) \geq0$, for all $x, x^\prime \in \Updelta$, 
equivalent to $y^\top  (\Phi  +\Phi^\top ) y\leq 0$ for all $y \in  T\Updelta$, or $\Phi + \Phi^\top $ is negative semidefinite with respect to $T \Updelta$. 
This is met for  example in zero-sum games, where $\mathbb{B}= - \mathbb{A}$, hence $\Phi$ is skew-symmetric and $x^\top  \Phi x =0$, for all $x \in \Updelta$.
Another class is the class of concave potential games where the payoff vector $U(x)$ can be expressed as the gradient of a $C^1$, concave potential function $P$,  as in congestion games, \cite{Sandholm}, \cite{Cominetti}, \cite{Chasparis_and_Shamma}. In Section  VI  we consider several examples. 
\end{remark}


\begin{thm}
	\label{thm:stable_1}
Let  $\mathcal{G}$ be a finite game with players' learning schemes as given by  EXP-D-RL,  \eqref{eqn:first_order_score_dyn} or  \eqref{eqn:first_order_exponentially_discounted_dynamics_system} (\autoref{fig:Z_Exp}). Assume there are a  finite number of isolated
 fixed-points  $\overline{z}^\star$ of $U\circ \sigma$. 
Then,  under  \autoref{assump:payoff_monotonicity}(i), for any $\epsilon>0$, players' scores  $z(t)=(z^p(t))_{p \in \mathcal{N}}$ 
 converge to 
a rest point $\overline{z}^\star$, 
while players' strategies  $x(t) =(x^p(t))_{p \in \mathcal{N}}$, $x(t) =\bm{\sigma}(z(t))$  
converge to a Nash distribution $\overline{x}^\star= \bm{\sigma}(\overline{z}^\star)$ of $\mathcal{G}$. 
Under \autoref{assump:payoff_monotonicity}(ii), the same conclusions hold for  any $ \epsilon>\mu$. 
\end{thm}

\begin{proof}
The proof is based on the representation of  \eqref{eqn:first_order_exponentially_discounted_dynamics_system} 
 in  \autoref{fig:Z_Exp} and passivity properties of  $\Sigma$ \eqref{eq:sigma_u} in Proposition \ref{prop:Sigma_OSEIP}.

First, recall that 
$\Omega$  is a compact, positively invariant set. Let $\overline{z}^\star$ be a rest point of \eqref{Exponential_Discount_Dynamic_2_Overall}, 
     $\overline{z}^\star = \overline{u}^\star=U(\bm{\sigma}(\overline{z}^\star))$,  $ \overline{x}^\star= \bm{\sigma}(\overline{z}^\star)$, (cf. \eqref{eqn:fixed_point_condition_score_dyn}) is a Nash distribution. Consider as  Lyapunov function 
the  storage function $V$, \eqref{eq:V_storage} at $\overline{z}^\star$, i.e., \vspace{-0.2cm}
\begin{equation}\label{eqn:V_thm1}
	V_{\overline{z}^\star}\!(z) \!\!= \!\textstyle{\sum}_{p \in \mathcal{N}} \!\Big (\! \lse^p(z^p) - \lse^p({\overline{z}^p}^\star) - \nabla\lse^p({\overline{z}^p}^\star)^\top (z^p-{\overline{z}^p}^\star) \! \Big)
	\end{equation}
	where 
	$V_{\overline{z}^\star}(\overline{z}^\star) = 0$, and 
	$V_{\overline{z}^\star}(z) \geq 0, \forall z \in \mathbb{R}^n$. Moreover, using \eqref{def:log_sum_exp}, $\nabla\lse^p = \sigma^p$ and $\sum_{j \in \mathcal{A}^p} \sigma^p_j(z^p) = 1$, it can be shown that $V_{\overline{z}^\star}\!(\overline{z}^\star + \mathbf{1}c) =0$,  $\forall  c \in \mathbb{R}$, so $V_{\overline{z}^\star}(\cdot)$ is positive semidefinite, but not positive definite. 
Take the time derivative of $V_{\overline{z}^\star}\!(z)$ along solutions of \eqref{Exponential_Discount_Dynamic_2_Overall}. Then,  \eqref{OSEIP_Sigma} holds and setting  $\overline{u} = \overline{u}^\star$ and  $\overline{z} = \overline{z}^\star$, yields \vspace{-0.2cm}
	\begin{equation}\label{OSEIP_Sigma_star}
	\dot V_{\overline{z}^\star}\!(z)  \leq   - \gamma\epsilon \|\bm{\sigma}(z) - \bm{\sigma}(\overline{z}^\star)\|_2^2 + \gamma(\bm{\sigma}(z) - \bm{\sigma}(\overline{z}^\star))^\top (u - \overline{u}^\star).
\end{equation}
By \eqref{eq:sigma_u}, \eqref{eqn:fixed_point_condition_score_dyn},   $u = U (\bm{\sigma}(z))$, 
$\overline{u}^\star = U (\bm{\sigma}(\overline{z}^\star))$ and using 
\autoref{assump:payoff_monotonicity}(i) for $x=\bm{\sigma}(z)$, $\overline{x}^\star=\bm{\sigma}(\overline{z}^\star)$, it follows that the second term in \eqref{OSEIP_Sigma_star} is non-positive, hence \vspace{-0.2cm}
	\begin{align}
	\dot V_{\overline{z}^\star}\!(z)  \leq &  - \gamma\epsilon \|\bm{\sigma}(z) - \bm{\sigma}(\overline{z}^\star)\|_2^2. 
	\numberthis
	\label{eqn:thm1_ineq}
	\end{align}
Thus for any 	$ \epsilon> 0$, $\dot V_{\overline{z}^\star}\!(z) \leq 0, \forall z \in \mathbb{R}^n$, and $\dot V_{\overline{z}^\star}\!(z) = 0$, for all $z \in \mathcal{E} = \{ z \in \Omega | \bm{\sigma}(z) = \bm{\sigma}(\overline{z}^\star) \}$. We apply LaSalle's invariance principle, \cite{Khalil}, and find the largest invariant subset $\mathcal{M}$ of $\mathcal{E}$ for \eqref{Exponential_Discount_Dynamic_2_Overall}. On $\mathcal{E}$ the dynamics of \eqref{Exponential_Discount_Dynamic_2_Overall} reduce to \vspace{-0.2cm}
\begin{align*}
\dot z = \gamma \Big ( U(\bm{\sigma}(\overline{z}^\star))  - z \Big) = \gamma (\overline{z}^\star -z).
\end{align*}
Since $\gamma>0$,  $z(t) \rightarrow \overline{z}^\star$ as $t \rightarrow \infty$, for any $z(0) \in \mathcal{E}$. Thus, no other solution except   $\overline{z}^\star$ can stay forever in $\mathcal{E}$, and   $\mathcal{M}$ consists only of equilibria. 
Since by assumption \eqref{Exponential_Discount_Dynamic_2_Overall} has a finite number of isolated equilibria $\overline{z}^\star$, by LaSalle's invariance principle, \cite{Khalil}, it follows that for any $z(0) \in \Omega$, $z(t)$ converges to one of them.  
By continuity of   $\bm{\sigma}$, it follows that $x(t) = \bm{\sigma}(z(t))$ converges to $\bm{\sigma}(\overline{z}^\star)$ which is a Nash distribution. 

Alternatively, from \eqref{eqn:thm1_ineq} since $V_{\overline{z}^\star}\!(z(t))$ is non-increasing and bounded from below by $0$, it converges as $t \to \infty$. Furthermore,  $\!\int\limits_0^\top   \!\|\bm{\sigma}(z(\tau))\! - \!\bm{\sigma}(\overline{z}^\star)\|_2^2 \mathrm{d}\tau \!\leq \!\dfrac{1}{\gamma \epsilon}[V_{\overline{z}^\star}\!(z(0))\! -\! V_{\overline{z}^\star}\!(z(t))]$. Since $V_{\overline{z}^\star}\!(z(t))$ converges, by the comparison principle it follows that  $\lim\limits_{t \to \infty} \int\limits_0^\top   \|\bm{\sigma}(z(\tau)) - \bm{\sigma}(\overline{z}^\star)\|_2^2 \mathrm{d}\tau = \lim\limits_{t \to \infty} \int\limits_0^\top   \|x(\tau) - \overline{x}^\star\|_2^2 \mathrm{d}\tau$ exists (finite). Since $x(t)$ is bounded, $\dot x(t)$ is uniformly bounded in $t$ and $x(t)$ is uniformly continuous. From Barbalat's lemma \cite{Khalil}, we conclude that  $x(t)= \bm{\sigma}({z}(t)) $ converges to $\overline{x}^\star=\bm{\sigma}(\overline{z}^\star)$ as $t \to \infty$.
	
Under \autoref{assump:payoff_monotonicity}(ii),  from \eqref{OSEIP_Sigma_star},  instead of \eqref{eqn:thm1_ineq}, the following inequality is obtained \vspace{-0.2cm}
\begin{align*}
	\dot V_{\overline{z}^\star}\!(z)  \leq &  - \gamma\epsilon \|\bm{\sigma}(z) - \bm{\sigma}(\overline{z}^\star)\|_2^2 +  \gamma\mu \|\bm{\sigma}(z) - \bm{\sigma}(\overline{x}^\star)\|_2^2,
	\label{eqn:thm1_ineq2}
	\end{align*}
or, $\dot V_{\overline{z}^\star}\!(z)  \leq  - \gamma(\epsilon -\mu) \|\bm{\sigma}(z) - \bm{\sigma}(\overline{z}^\star)\|_2^2$. 
Thus, 
for any $\epsilon >\mu$, \vspace{-0.2cm}
this is as \eqref{eqn:thm1_ineq} and the proof can follow  as in the above. 
\end{proof}
\begin{remark}
%
In general, convergence is to the set of  Nash distributions.  This is the case of potential games, where multiple (logit) equilibria can exist. This is the same as 
for the smooth best-response dynamics  (cf. Theorem 3.3 \cite{HopkinsHasy}). On the other hand, any 2-player zero-sum game has a unique Nash distribution, (cf. Theorem 3.2, in \cite{HopkinsHasy}). Thus, Theorem \ref{thm:stable_1}  is similar to  results for the perturbed best-response dynamics, (cf. Theorem 3.3 and 3.2, in \cite{HopkinsHasy}), or Q-learning (Proposition 4.2 in \cite{Leslie}),  but is valid for $N$-player games.   
We note that in \cite{Coucheney},  based on analyzing the primal (induced) strategy dynamics,  Proposition 3.4 shows that   $x(t)$ converges to a  logit equilibrium of $\mathcal{G}$ in potential games, while Theorem 3.8 shows that 
perturbed \textit{strict} Nash equilibria are asymptotically stable for small enough $T$. 
Theorem \ref{thm:stable_1} analyzes  the dual (score) dynamics and shows that under \autoref{assump:payoff_monotonicity}(i) (monotonicity of  $-U$), players converge arbitrarily close to a Nash equilibrium (not necessarily \textit{strict}), by taking small $\epsilon$ (cf. Proposition \ref{prop:perturbed_Nash_equilibria}).
 Under hypo-monotonicity, \autoref{assump:payoff_monotonicity}(ii), 
convergence to a Nash distribution is 
guaranteed for  $\epsilon>\mu$, based on balancing the shortage of passivity of $-U$ by the excess passivity (OSEIP) of $\Sigma$. Our results can be also applied to monocyclic games, \cite{BenaimHofbauerHopkins2009}, which only have mixed NEs and which generalize the Rock-Paper-Scissor game (see Section VI). Proposition 3 in \cite{BenaimHofbauerHopkins2009} shows that in an unstable game ($\mathbb{A}$ is positive  definite) a completely mixed perturbed NE (or an NDE) is unstable for sufficiently small $\epsilon$ under the smooth best-response dynamics.  Based on  \autoref{assump:payoff_monotonicity}(ii), our result in Theorem \ref{thm:stable_1} characterizes the range of $\epsilon$ that guarantees convergence to a Nash distribution under EXP-D-RL dynamics in an unstable game.  Of course, too large an $\epsilon$ moves the Nash distribution away from the Nash equilibrium, so a trade-off exists. Several examples are discussed in Section VI.

\end{remark}

\begin{remark}
Theorem \ref{thm:stable_1} is related to passivity-based feedback (PBF) output regulation \cite{BIW1991}. 
 In our case,  the output-strictly EIP system $\Sigma$ has the output $x=\bm{\sigma}(z)$ regulated to $\bm{\sigma}(\overline{z}^\star)$ instead of  to the origin, and the static output feedback is   $u = U(x)$, which itself is PBF for any payoff $U$ satisfying \autoref{assump:payoff_monotonicity}(i). In fact, $\Sigma$ is $\overline{z}^\star$-detectable, and $z(t)$ converges to $\overline{z}^\star$ (since $V_{\overline{z}^\star}$  is positive semidefinite, we can show only attractivity).  We also note that a similar approach can be considered  for the un-discounted EXP-RL in \cite{Merti_Learning},\cite{Laraki}, for which $G(s) = \frac{1}{s}$ and  $\Sigma$ is only lossless EIP. Then convergence is guaranteed for games with payoff $U$ satisfying a strict version of \autoref{assump:payoff_monotonicity}(i), (strictly stable games in \cite{Sandholm}).
 \end{remark}

\vspace{-0.2cm}
\section{Passivity-Based Higher-Order EXP-D-RL}

In this section, we propose a class of higher-order extension of the EXP-D-RL scheme \eqref{eqn:first_order_score_dyn},  based on passivity arguments. 

Suppose instead that during game play, each player $p$ continuously aggregates his own strategy $x^p \in \Updelta^p$ over time into an \textit{adjustment variable} $v_a^p \in {\mathbb{R}^n}^p$, as given by $\dot{\xi}^p = A \xi^p + B x^p$, $v_a^p = C \xi^p + D x^p$.  After receiving his payoff, the player subtracts the variable from the payoff $u^p$ to form an adjusted payoff  $\tilde u^p = u^p - v_a^p$ and aggregates the adjusted payoff $\tilde u^p$ into a score $z^p$. Thus player $p$ (H-EXP-D-RL) learning scheme is given by \vspace{-0.2cm}
	\begin{equation}
	\begin{cases}
	\dot z^p &= \gamma (-z^p + U^p(x) - v^p_a), \,\, z^p(0) \in \mathbb{R}^{n^p}\\ 
	\dot \xi^p &= A\xi^p+B x^p\\
	v^p_a &= C \xi^p + D x^p\\
	x^p &= \sigma^p(z^p),
	\end{cases}
	\label{eqn:closed_loop_feedback_system_2_p}
	\end{equation}
with $\xi^p(0) =0$. Let $\tilde{u} \coloneqq u - v_a$ denote the overall adjusted payoff, where $v_a$ denotes the overall adjustment. 
Then   the overall  
players' learning dynamics  \eqref{eqn:closed_loop_feedback_system_2_p} is given as \vspace{-0.2cm}
\begin{equation}
\begin{cases}
\dot z &= \gamma({-z + u - v_a}), \,\, z \in \mathbb{R}^n, \xi(0) \in \mathbb{R}^n \\
\dot \xi &= \A\xi+\B x\\
v_a &= \C \xi + \D x\\
x &= \bm{\sigma}(z)\\
u &= U(x), 
\end{cases}
\label{eqn:closed_loop_feedback_system}
\end{equation}
where $\A$, $\B$, $\C$, $\D$ are block-diagonal matrices with $A$,$B$,$C$,$D$ on the diagonal, respectively. Note that the modified H-EXP-D-RL scheme  \eqref{eqn:closed_loop_feedback_system} can be represented as in \autoref{fig:interconnection_1},  as a feedback interconnection between $\tilde{\Sigma}$ and $U$,  a feedback modification to EXP-D-RL in \autoref{fig:Z_Exp}  where, \vspace{-0.2cm}
 \begin{equation} \mathbf{H}_a(s) \coloneqq \mathbf{C}(s\mathbf{I}_{n}- \A)^{-1}\B + \D \label{eqn:H_a_matrix}, \end{equation}
is as a payoff-adjustment learning rule.

\begin{figure}[ht]
\vspace{-0.2cm}
	\begin{center}
		\includegraphics[scale=0.6]{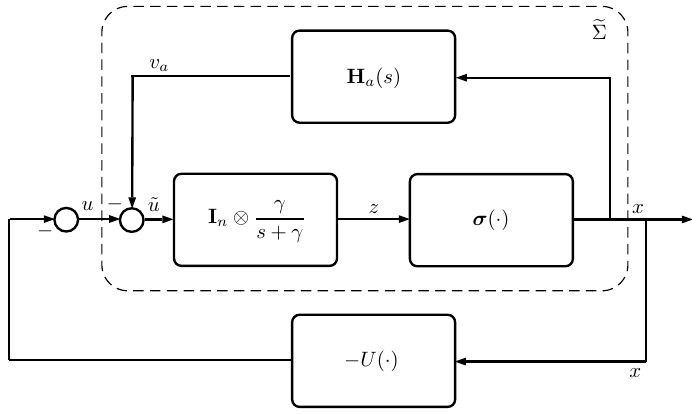}
	\end{center}
	\caption{H-EXP-D-RL  \eqref{eqn:closed_loop_feedback_system_2_p} as feedback interconnection}
	\label{fig:interconnection_1}
	\vspace{-0.3cm}
\end{figure}
In \autoref{fig:interconnection_1},  the forward path between $\tilde{u}$ and $x$ is $\Sigma$ as in  \eqref{eq:sigma_u},   which is OSEIP by Proposition \ref{prop:Sigma_OSEIP}. This means that feedback interconnection with another EIP system $ \mathbf{H}_a$ will be preserve passivity properties. This is used next. 

\begin{thm}
	\label{prop:exponentially_stable_prop}
Let  $\mathcal{G}$ be a finite game with players' higher-order learning schemes as given by H-EXP-D-RL,   \eqref{eqn:closed_loop_feedback_system_2_p}, or \eqref{eqn:closed_loop_feedback_system} represented as in \autoref{fig:interconnection_1}. 
 Assume there are a  finite number of isolated  fixed-points $\overline{z}^\star$ of $U\circ \sigma$. 
 Assume that \eqref{eqn:H_a_matrix} is proper, strictly positive real, 
 and $\mathbf{H}_a(0) = \mathbf{0}_n$. 
Then, under  \autoref{assump:payoff_monotonicity}(i), for any $\epsilon>0$, players' scores  $z(t)=(z^p(t))_{p \in \mathcal{N}}$ 
 converge to 
a  $\overline{z}^\star$, 
while players' strategies  $x(t) =(x^p(t))_{p \in \mathcal{N}}$, $x(t) =\bm{\sigma}(z(t))$  
converge to a Nash distribution $\overline{x}^\star= \bm{\sigma}(\overline{z}^\star)$ of $\mathcal{G}$. 
Under \autoref{assump:payoff_monotonicity}(ii), the same conclusions hold for  any $ \epsilon>\mu$. 
\end{thm}

\begin{proof}
	At an equilibrium condition  $\!(\overline{\xi}^\star, \!\overline{z}^\star, \!\overline{x}^\star, \!\overline{u}^\star,\!\overline{v}_a^\star)\!$ of 
	\eqref{eqn:closed_loop_feedback_system}, \vspace{-0.2cm}
\begin{equation}\label{eqn:eq_II}
\begin{array}{ll}
		\mathbf{0} &= \gamma(-\overline{z}^\star + \overline{u}^\star  - \overline{v}_a^\star)\\ 
		\mathbf{0} &= \A\overline{\xi}^\star+ \B \overline{x}^\star, \qquad \overline{v}_a^\star = \mathbf{C}\overline{\xi}^\star + \D \overline{x}^\star\\
		\overline{x}^\star &= \bm{\sigma}(\overline{z}^\star), \quad \qquad  \qquad \overline{u}^\star = U(\overline{x}^\star). 	
\end{array}
\end{equation}
From the first and last line, $\overline{z}^\star = U (\bm{\sigma}(\overline{z}^\star)) - \overline{v}_a^\star$.	Since $\mathbf{H}_a$ is strictly positive real, $\A$ is Hurwitz, hence is invertible and from the second  line we obtain, $\overline{\xi}^\star = -\A^{-1}\B \overline{x}^\star$ and $\overline{v}_a^\star  = (- \mathbf{C}\A^{-1}\B + \D)\overline{x}^\star$. By assumption $\mathbf{H}_a(0) = -\mathbf{C}\A^{-1}\B + \D = \mathbf{0}_{n}$, so  that  $\overline{v}_a^\star = \mathbf{0}$, and therefore \eqref{eqn:eq_II}  reduces to \eqref{eqn:fixed_point_condition_score_dyn}. Thus, under H-EXP-D-RL, equilibria points $\overline{z}^\star$ under EXP-D-RL are preserved,  
 and  
$\overline{x}^\star$ is a Nash distribution of $\mathcal{G}$. To prove convergence of $x(t)$ to $\overline{x}^\star$ we leverage  passivity properties of the subsystems  in \autoref{fig:interconnection_1}. 
Specifically,  by \autoref{assump:payoff_monotonicity}(i), $-U$ is incrementally passive, $\Sigma$ from $\tilde u$ to $x$ is OSEIP (cf. Proposition \ref{prop:Sigma_OSEIP}) and $\mathbf{H}_a$ is EIP  passive. 

	Consider the positive semidefinite storage function of $\Sigma$, $V_{\overline{z}^\star}\!(z)$, \eqref{eqn:V_thm1}, as in  the proof of \autoref{thm:stable_1}. 
	 Taking the time derivative of $V_{\overline{z}^\star}\!(z)$ along the solutions $z(t)$ in  \eqref{eqn:closed_loop_feedback_system}, using the cocoercivity of  $\bm{\sigma}$ and following the proof of Proposition \ref{prop:Sigma_OSEIP}, it can be shown  that $\dot V_{\overline{z}^\star}\!(z) =\nabla V_{\overline{z}^\star}\!(z)^\top   \! \dot{z} $ satisfies \vspace{-0.2cm}
	\begin{align}\label{eq_V_thm2}
	\dot V_{\overline{z}^\star}\!(z) \leq& -\gamma \epsilon\|\bm{\sigma}(z) - \bm{\sigma}(\overline{z}^\star)\|_2^2 + \gamma(\bm{\sigma}(z) \!-\! \bm{\sigma}(\overline{z}^\star))^\top \!\! (u - \overline{u}^\star)\nonumber\\ 
	& - \gamma(\bm{\sigma}(z) \!-\! \bm{\sigma}(\overline{z}^\star))^\top \!\! (v_a - v_a^{*}).
	\end{align} 
Since  $\mathbf{H}_a(s)$ is strictly positive real, it is strictly passive and has a quadratic storage function $V_a(\xi) = \dfrac{\gamma}{2}(\xi-\overline{\xi}^\star)^\top  \mathbf{P}(\xi-\overline{\xi}^\star),$ where $\P$ is a   positive definite matrix  as given by the KYP Lemma, \cite{Khalil}, and the time derivative of $V_a$ along $\xi(t)$, $\dot V_a(\xi) = \nabla V_a(\xi)^\top  \dot \xi$, satisfies, \vspace{-0.2cm}
\begin{equation}\label{eq_V_a_thm2}
\dot V_a(\xi) \leq  \gamma(v_a-\overline{v}_a^\star)^\top  (x - \overline{x}^\star).
\end{equation}
	Consider the composite Lyapunov function $W(z, \xi) \coloneqq  V_{\overline{z}^\star}\!(z) + V_a(\xi)$. By construction, $W(z, \xi)$ is positive semidefinite. Taking the time derivative along the solution trajectories of \eqref{eqn:closed_loop_feedback_system}, $\dot W(z, \xi)  =  \dot V_a(\xi) + \dot V_{\overline{z}^\star}\!(z)=\nabla V_{\overline{z}^\star}\!(z)^\top   \! \dot{z} + \nabla  V_a(\xi) \dot \xi$, and 
	using \eqref{eq_V_thm2}, \eqref{eq_V_a_thm2} and $\bm{\sigma}(z) = x, \bm{\sigma}(\overline{z}^\star) = \overline{x}^\star$, yields, \vspace{-0.2cm}	\begin{align*}
	\dot W(z, \xi)  = & \dot V_a(\xi) + \dot V_{\overline{z}^\star}\!(z)\\ 
  & \hspace{-1cm}  \leq -\gamma \epsilon \|\bm{\sigma}(z) \!-\! \bm{\sigma}(\overline{z}^\star)\|_2^2\! + \!\gamma(\bm{\sigma}(z) \!- \!\bm{\sigma}(\overline{z}^\star))^\top  (u - \overline{u}^\star),
	\end{align*}
which is similar to \eqref{OSEIP_Sigma_star}. 
From this point, the proof follows using the same arguments as in the proof of \autoref{thm:stable_1}, with $\mathbf{A}$  Hurwitz being used to show that the largest invariant subset of  \eqref{eqn:closed_loop_feedback_system}  is $\mathcal{M} = \{ (\overline{z}^\star, \overline{\xi}^\star) \}$, where $\overline{\xi}^\star =-\mathbf{A}^{-1} \mathbf{B} \bm{\sigma}(\overline{z}^\star)$.  
		\end{proof}

\begin{remark}	
We note that the higher-order dynamics proposed here, designed based on passivity principles, is guaranteed to maintain the same properties of convergence in monotone (stable) games as the first-order dynamics. This is unlike the  higher-order dynamics proposed in \cite{Laraki} based on second-order ($n$-th order) integration of payoffs, which, as shown in \cite{Mabrok}, are no longer passive and for which convergence in  stable games is no longer guaranteed.

A specific second-order score dynamics from this class of RL schemes can be obtained for 
	\vspace{-0.2cm}
	\begin{equation}
	G_a(s) = \dfrac{Ks}{s+a},
	\label{eqn:pr_example_1}
	\end{equation}
	where  $K,a > 0$.  $\mathbf{H}_a(s) = \mathbf{I}_n \otimes G_a(s)$ satisfies the conditions in Theorem \ref{prop:exponentially_stable_prop} (by Lemma 6.1 in \cite{Khalil}). 
Then 
$A=-a \mathbf{I}_{n^p}$, $B=-a \mathbf{I}_{n^p}$, $C=D=K\mathbf{I}_{n^p}$,  and H-EXP-D-RL  \eqref{eqn:closed_loop_feedback_system_2_p} is  \vspace{-0.2cm}
	\begin{equation}
	\begin{cases}
	\dot z^p &= \gamma (-z^p + u^p- K(  \xi^p +  x^p)),  \,\, z^p(0) \in \mathbb{R}^{n^p}\\ 
	\dot \xi^p &= -a\xi^p-a x^p, \, \, \, \, \, \, \, \,  \xi(0)^p = \mathbf{0}\\
	x^p &= \sigma^p(z^p),
	\end{cases}
	\label{eqn:closed_loop_feedback_system_2_p_2nd}
	\end{equation}
where $u^p=U^p(x) $. Note that \eqref{eqn:pr_example_1} represents a high-pass filter. We show in Section VII that for cyclical games such as the Rock-Paper-Scissors game, such a scheme 
can in fact foster convergence. 
\end{remark}


\section{Connection to population games}


In this section we offer a population game interpretation  for the \textit{ induced} mixed-strategy dynamics of the reinforcement learning scheme. 
A  population game involves  a single (or multiple) large population(s) of agents interacting, whose members are anonymous. The agents are repeatedly randomly matched to play a symmetric two-player normal form game defined by strategy set $\mathcal{A}^p$ and a payoff matrix $\mathbb{A}$, or they could play ``the field", as in congestion games, \cite{Sandholm}. 
The identification between a large population game and continuous-time reinforcement learning is based on the fact that the dynamical equations governing the evolution of population states and those governing the evolution of mixed strategies are the same.

Consider first the induced strategy dynamics  under EXP-D-RL, obtained as the set of ODEs that describe the evolution $x(t)$ on $\Updelta$, as induced by the score dynamics  \eqref{eqn:first_order_score_dyn} on $\mathbb{R}^n$. 
Note that based on $x^p_i = \dfrac{\exp(\epsilon^{-1} z^p_i)}{\textstyle{\sum}_{k \in \mathcal{A}^p} \exp(\epsilon^{-1} z^p_k)}$ for all $i \in \mathcal{A}^p$, 
we obtain for all $i, j \in \mathcal{A}^p$
\begin{align*}
\log(x^p_i) - \log(x^p_j) & = \epsilon^{-1} (z^p_i - z^p_j) \numberthis \label{eqn:diff_ln_x_z}\end{align*} Taking the time derivative on both sides yields 
$x^p_j\dot x^p_i   = x^p_i\dot x^p_j  + x^p_ix^p_j\epsilon^{-1} (\dot z^p_i - \dot z^p_j)$. 
Summing this over all $j \in \mathcal{A}^p$ on both sides and using $\sum_{j \in \mathcal{A}^p} x_j^p = 1$ and $\sum_{j \in \mathcal{A}^p} \dot x_j^p = 0$ yields $ \dot x^p_i   = \epsilon^{-1} x^p_i(\dot z^p_i - \textstyle{\sum}_{j \in \mathcal{A}^p}x^p_j\dot z^p_j)$. Based on EXP-D-RL, substituting in $\dot z^p_i$ and $\dot z^p_j$ from \eqref{eqn:first_order_score_dyn} into the last expression, we obtain:
\begin{equation}
\dot x^p_i = {\gamma} \epsilon^{-1} [x^p_i(u^p_i - \textstyle{\sum}_{j \in \mathcal{A}^p}x^p_j u^p_j)- x^p_i(z^p_i - \textstyle{\sum}_{j \in \mathcal{A}^p}x^p_j z^p_j)]
\label{eqn:first_order_induced_dyn} 
\end{equation}
where $u^p_i =U^p_i(x)$ and $z^p_i$ is generated by EXP-D-RL. 
An alternative form is obtained by 
multiplying  \eqref{eqn:diff_ln_x_z} by $x^p_j$, summing over $j \in \mathcal{A}^p$ and using $\sum_{j \in \mathcal{A}^p} x_j^p = 1$, which  yields:
\[
\log(x^p_i) - \textstyle{\sum}_{j \in \mathcal{A}^p}x^p_j\log(x^p_j) = \epsilon^{-1} (z^p_i - \textstyle{\sum}_{j \in \mathcal{A}^p}x^p_jz^p_j)
\]
Substituting the above expression into \eqref{eqn:first_order_induced_dyn}, we obtain
\begin{equation}
\dot x^p_i =  {\gamma} \epsilon^{-1} x^p_i(u^p_i - \textstyle{\sum}_{j \in \mathcal{A}^p}x^p_ju^p_j)+ {\gamma}x^p_i(\textstyle{\sum}_{j \in \mathcal{A}^p}x_j^p\log(\frac{x^p_j}{x^p_i}))
\label{eqn:selection-mutation-entropy}
\end{equation}

Next, we study connections between \eqref{eqn:first_order_induced_dyn} or \eqref{eqn:selection-mutation-entropy} and existing population dynamics.  Consider  a single (or multiple) population(s) of agents. Each agent in population $p$ is  preprogrammed to play a certain pure strategy $i \in \mathcal{A}^p$.   In this case, $x^p_i$ is identified as  the fraction of population $p$ using strategy $i \in \mathcal{A}^p$.  Furthermore,  $x^p$ is a population state (corresponding to the mixed strategy of player $p$), and  $x$ refers to the overall population state. In population games generated by (random) matching in normal form games  we can identify each population $p$ with a player.   We note that the single-player reinforcement learning (typically referred to as a ``play against nature") is equivalent to a single-population matching.  The evolution of population states in a population game can be modelled using a switching rule called the \textit{revision protocol}, \cite{Sandholm}. Formally, a revision protocol is a map $\rho^p: \Updelta^p \times \mathbb{R}^{n^p} \to \mathbb{R}^{{n}^p \times {n}^p}_{\geq 0}$ where $\rho^p_{ij} \in \mathbb{R}_{\geq 0}$ is referred to as the \textit{conditional switch rate},  proportional to the probability of agents in population $p$ playing pure strategy $i$ that switch to using strategy $j$. The flow of population share of agents under a particular revision protocol is given by,\cite{Sandholm},
\begin{equation}
\dot x^p_{i} = \sum\limits_{j \in \mathcal{A}^p} x^p_{j} \rho^p_{ji} - x^p_{i} \sum\limits_{j \in \mathcal{A}^p} \rho^p_{ij},
\label{eqn:mean_dynamics} \tag{MD}
\end{equation} 
which is referred as the \textit{mean dynamics}. The first term of \eqref{eqn:mean_dynamics} represents the rate of inflow of agents adopting strategy $i$ and the second term represents the rate of outflow of agents abandoning it. We note that when $\rho^p_{ij} = x^p_{j}(u^p_j - Z)$, where  $Z \in \mathbb{R}$ acts as a threshold constant that is less than any payoff $u^p_j$, $\rho^p_{ij}$  is the so called ``imitation of success" revision protocol \cite{Sandholm}, which generates the \emph{replicator dynamics} as (MD). It is well known that the replicator dynamics converges in \emph{strictly} stable games, but cycles  in (null) stable games.  
Next, we show that the revision protocol\footnote{Strictly speaking, the proposed revision protocol is well defined only when $z_j^p$ is less than $u_j^p$ and $u_j^p \geq 0$ for all $j \in \mathcal{A}^p$. The proposed revision protocol here only serves to provide a behavioural intuition to the induced dynamics.}:
\begin{equation}
\rho^p_{ij} = {\gamma}\epsilon^{-1} x^p_{j}(u^p_j -z^p_j)
\label{def:imitation_of_discounted_success}
\end{equation}
where $z^p_j$ is updated as in \eqref{def:exponential_discount_learning}, generates \eqref{eqn:first_order_induced_dyn}. Indeed, using \eqref{eqn:mean_dynamics}, \eqref{def:imitation_of_discounted_success} and $\sum_{j \in \mathcal{A}^p} x_j^p = 1$ yields:
\begin{equation*}
\begin{split}
\dot x^p_{i} &= {\gamma}\epsilon^{-1}(\textstyle{\sum}_{j \in \mathcal{A}^p} x^p_{j} x^p_{i}(u^p_i -z^p_i) - x^p_{i}\textstyle{\sum}_{j \in \mathcal{A}^p} x^p_{j}(u^p_j -z^p_j))\\
& = {\gamma}\epsilon^{-1} (x^p_i(u^p_i - \textstyle{\sum}_{j \in \mathcal{A}^p}x^p_j u^p_j)- x^p_i(z^p_i - \textstyle{\sum}_{j \in \mathcal{A}^p}x^p_j z^p_j))
\end{split}
\end{equation*} which is \eqref{eqn:first_order_induced_dyn}. In light of the similarity of   \eqref{def:imitation_of_discounted_success} with the ``imitation of success" revision protocol,  we refer to \eqref{def:imitation_of_discounted_success} as ``imitation of exponentially-discounted success". 
 We note that \eqref{eqn:first_order_induced_dyn} takes on the form of (scaled) replicator dynamics \cite{Sandholm} with an extra \emph{replicator-like correction term } that depends on the \textit{excess score}, which is precisely the ``penalty-adjusted replicator equation" with discount rate $T= 1$ \cite{Coucheney}. Furthermore, \eqref{eqn:selection-mutation-entropy} is simply the (scaled) replicator dynamics under the influence of a relative-entropy term that is independent of the game payoff rate. We note that if the score is  a direct integration of the payoff, i.e.,  $\dot z^p_j = u^p_j$,  then we recover the replicator dynamics obtained by \cite{Laraki} under exponential learning. It is the  replicator-like correction term in  \eqref{eqn:first_order_induced_dyn} (induced by discounted integration) that enables convergence in merely stable games, not necessarily strict. This beneficial property is due by the more sophisticated imitation protocol  \eqref{def:imitation_of_discounted_success} versus pure imitation. An interpretation of \eqref{def:imitation_of_discounted_success} can be given as follows. When agents playing $i$ see that agent playing $j$ has nonzero payoff, they switch to strategy $j$ with a rate proportional to the difference between this payoff and the cumulative discounted payoff over time,  multiplied by the popularity strategy $j$. This means an agent is more likely to switch to strategy $j$ if the payoff $u^p_j$ keeps on getting better over time. 
 %
 We also note that similar to this, one can show that the induced dynamics of the higher-order reinforcement learning scheme (H-EXP-D-RL)  is generated  via the revision protocol, \vspace{-0.2cm}\begin{equation}\rho^p_{ij} = {\gamma} \epsilon^{-1} x^p_{j}(u^p_j - {v_a^p}_j-z^p_j)\end{equation}where $v^p_{ai}= K\xi^p_i  + K x^p_i$. 



\vspace{-0.2cm}
\section{Examples and Numerical Simulations}

In this section, we discuss several examples: single-population game ($p=1$) (Example 1 and 2), two-player games (Example 3, 4 and 5) and three-player games (Example 6 and 7). In addition some difficult games  are discussed: modified RPS game, \cite{BenaimHofbauerHopkins2009}) (Example 8), and a modified  asymmetric Jordan game, \cite{ShammaTAC2005}, (Example 9). In all cases we compare the performance of \eqref{eqn:first_order_score_dyn} versus that of  \eqref{eqn:closed_loop_feedback_system_2_p_2nd} in games characterized by their monotonicity property as given by   \autoref{assump:payoff_monotonicity}(i) or (ii). We take 
 $K = 1, a= 1$ and   
 an arbitrary initial condition $z(0)$. 
 
 In the first two examples we consider single-player reinforcement learning,  \cite[p.26]{Sutton},  typically referred to as a ``play against nature". The canonical example  is equivalent to a single-population matching,  \cite{Sandholm},  ($p\!=\!1$), where the identities of the agents can be interchanged. Under this identification,  consider a large population of agents, each randomly matched with an opponent to play a symmetric two-player game with $U(x) = \mathbb{A} x$. Then \eqref{eqn:U_Monotone_0}, $-(x-x^\prime)^\top  \mathbb{A} (x-x^\prime) \geq0$, for all $x, x^\prime \in \Updelta$, is  equivalent to 
$y^\top  \mathbb{A} y = \frac{1}{2}y^\top  (\mathbb{A} +\mathbb{A} ^\top ) y \leq 0$, for all $y \in T\Updelta$, and \autoref{assump:payoff_monotonicity} can be checked  via the eigenvalues of $\mathbb{A} +\mathbb{A} ^\top $. Thus, \autoref{assump:payoff_monotonicity}(i) holds if $\mathbb{A} +\mathbb{A} ^\top $ is negative semi-definite with respect to the tangent space $T\Delta$. Moreover, $  \frac{1}{2}y^\top  (\mathbb{A} +\mathbb{A} ^\top ) y \leq \frac{1}{2} \lambda_{\max}( \mathbb{A} +\mathbb{A} ^\top ) \|y\|^2$, for all $y \in T\Updelta$, where $\lambda_{\max} (\mathbb{A} +\mathbb{A} ^\top)$ is the maximum eigenvalue corresponding to an eigenvector in $T \Updelta$. Thus, \autoref{assump:payoff_monotonicity}(ii) holds with $\mu = \frac{1}{2} \lambda_{\max}( \mathbb{A} +\mathbb{A} ^\top ) $.

\begin{example} Single-population Rock-Paper-Scissors Game
\label{single_rps_game}

Consider a large population of agents, each randomly matched with an opponent to play the Rock-Paper-Scissors (RPS) game characterized by the  payoff matrix, \cite{Sandholm},\vspace{-0.1cm}
	\begin{equation}
	\label{eqn:RPS_game}
	\mathbb{A} = \begin{bmatrix*}[r] 0 & -l & 1 \\ 1 & 0 & -l \\ -l & 1 & 0 \end{bmatrix*}, 
	\end{equation}
where $l$ is  a parameter.
  For $l=1$ this known as a standard RPS game  (null stable, cf. \cite[p. 79]{Sandholm}). For $l<1$ the game is strictly stable,  while for $l >1$ it unstable (or bad RPS, cf. \cite{Sandholm}). This game has a unique, interior Nash equilibrium $x^\star$ which  coincides with the Nash distribution (logit equilibrium) $\overline{x}^\star = \begin{bmatrix}\frac{1}{3} & \frac{1}{3} & \frac{1}{3} \end{bmatrix}^\top $, and for which  $\overline{z}^\star = \frac{1-l}{3} \mathbf{1}$.  
	 The eigenvalues of $\mathbb{A} +\mathbb{A} ^\top $ are $\{2(1-l), l-1, l-1\}$ with the last two corresponding to eigenvectors in $T \Updelta$. Thus for $l\leq 1$, $\mathbb{A} +\mathbb{A} ^\top $  is  negative semi-definite with respect to the tangent space $T\Delta$ and satisfies \autoref{assump:payoff_monotonicity}(i).  	
Note that for all $x \in \Updelta$, 
$
	(x-x^\prime)^\top  \mathbb{A} (x-x^\prime)  = \dfrac{1}{2} (x-x^\prime)^\top  (\mathbb{A}+\mathbb{A}^\top  )(x-x^\prime)= \dfrac{l - 1}{2}\|x - x^\prime\|_2^2
$. 
Thus, \autoref{assump:payoff_monotonicity}(i) is satisfied for $l=1$ and $l \leq 1$ ((null) stable game) , while for  $l>1$ (unstable game),  \autoref{assump:payoff_monotonicity}(ii) holds with  $\mu = \frac{l-1}{2}$.   By \autoref{thm:stable_1} and \autoref{prop:exponentially_stable_prop}, if $l\leq 1$, \eqref{eqn:first_order_score_dyn} and \eqref{eqn:closed_loop_feedback_system_2_p_2nd} are guaranteed to converge for any $\epsilon >0$, while if $l>1$, convergence is guaranteed  for any  $\epsilon >\frac{l-1}{2} $. 
We note that this lower bound is not tight, and in fact  in simulations we see that convergence occurs even when it is violated. 
	We show simulation results for the Rock-Paper-Scissors game with $l = 1,2.5,5$ and $8$ in \autoref{fig:convergence_plot_rps_1}, \autoref{fig:convergence_plot_rps_2_5}, \autoref{fig:convergence_plot_rps_5} and \autoref{fig:convergence_plot_rps_8}, respectively, all for $\epsilon=1$. Each figure shows the score trajectories and  the induced strategy trajectories in single-population, where the blue line corresponds to trajectory of  \eqref{eqn:first_order_score_dyn}, while the red line corresponds to trajectory of \eqref{eqn:closed_loop_feedback_system_2_p_2nd}. 
	For $l = 1, 2.5, 5$ the score trajectories $z(t)$ of both \eqref{eqn:first_order_score_dyn} (blue) and \eqref{eqn:closed_loop_feedback_system_2_p_2nd} (red)  converge to the corresponding  $\overline{z}^\star$, $\mathbf{0}, -\frac{1}{2}\mathbf{1}, -\frac{4}{3}\mathbf{1}$, respectively, while the induced mixed-strategy $x(t)$ trajectories converge to $\overline{x}^\star$, with the second-order converging faster than the first-order dynamics. 
	For $l = 8$,  \eqref{eqn:closed_loop_feedback_system_2_p_2nd} converges, whereas \eqref{eqn:first_order_score_dyn} forms a limit cycle in the interior of the simplex. Computing the Jacobian matrix of the dynamics \eqref{eqn:first_order_score_dyn} at $\overline{z}^\star$ reveals that the eigenvalues are at $\{-1, \frac{l-1-6\epsilon}{6\epsilon} \pm \frac{ \sqrt{3} (1+l)}{6 \epsilon} i\} $, so a bifurcation occurs at $\epsilon=\frac{l-1}{6}$, which for $l=8$ is $\frac{7}{6}$. For  the Jacobian matrix of the dynamics \eqref{eqn:closed_loop_feedback_system_2_p_2nd}, a value of $\epsilon$ close to where the bifurcation occurs for $l=8$ is $0.86$ (computed numerically). 
	This explains why,  in the RPS game with $l = 8$,   for $\epsilon =1$, \eqref{eqn:closed_loop_feedback_system_2_p_2nd} converges, whereas \eqref{eqn:first_order_score_dyn} does not.  
	This example shows that the higher-order dynamics  \eqref{eqn:closed_loop_feedback_system_2_p_2nd} can potentially solve larger classes of games (``more unstable"). 
%
\vspace{-0.3cm}
	\begin{figure}[h]
\vspace{-0.2cm}
		\hfill
		\subfloat[Score trajectories]{\includegraphics[width=0.48\linewidth]{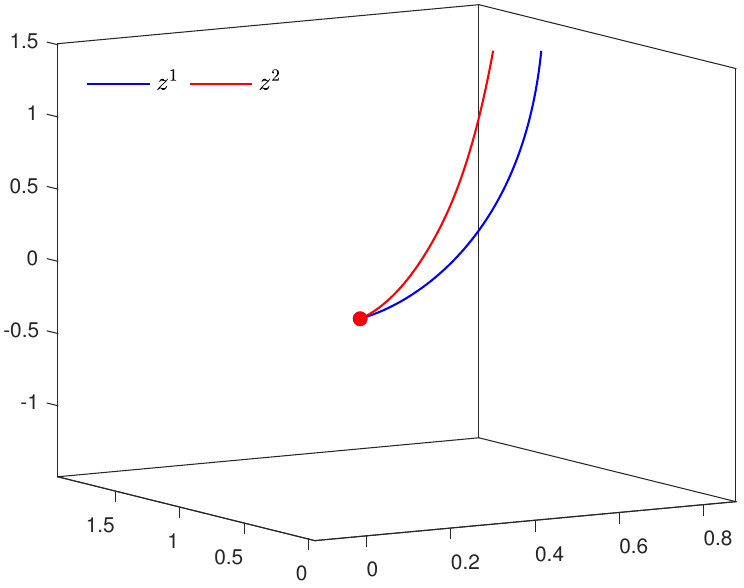}}
		\hfill
		\subfloat[Mixed strategy trajectories]{\includegraphics[width=0.5\linewidth]{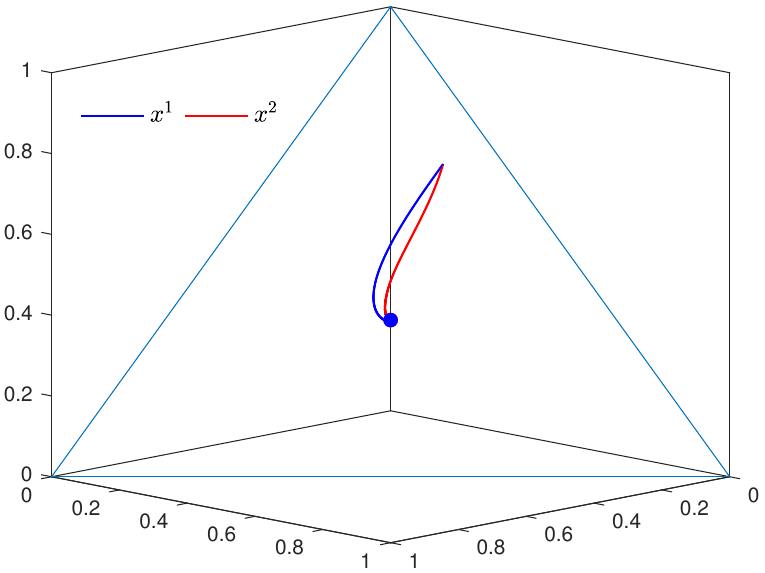}}
		\hfill
		\caption{Standard RPS game, $l = 1$, $\epsilon=1$}
		\label{fig:convergence_plot_rps_1}
\vspace{-0.3cm}
	\end{figure}
\vspace{-0.3cm}
\begin{figure}[h]
\vspace{-0.3cm}
		\hfill
		\subfloat[Score trajectories]{\includegraphics[width=0.48\linewidth]{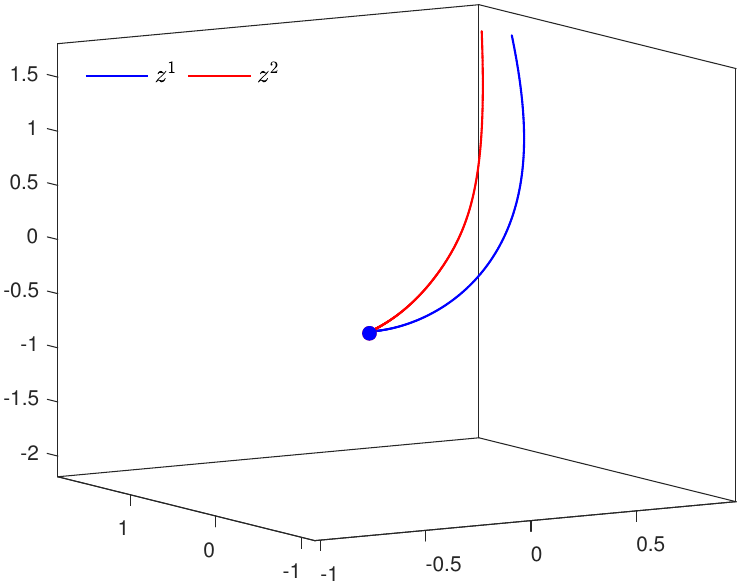}}
		\hfill
		\subfloat[Mixed strategy trajectories]{\includegraphics[width=0.5\linewidth]{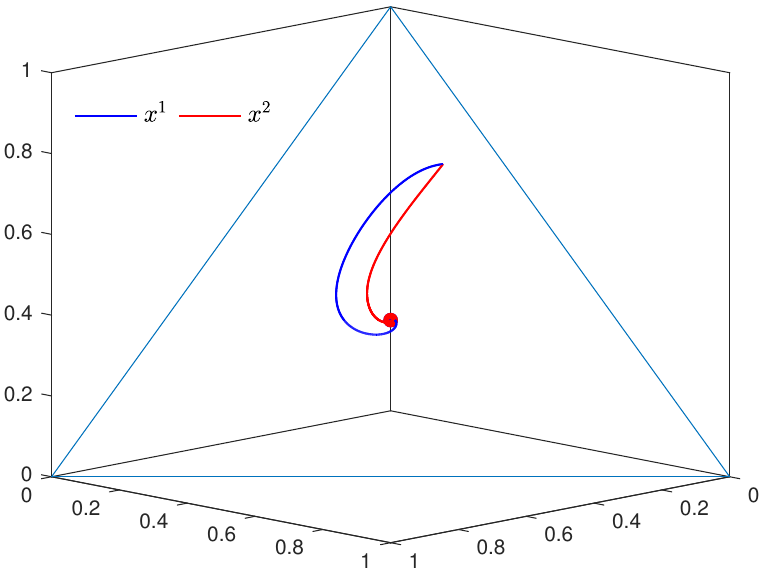}}
		\hfill
		\caption{Unstable RPS game, $l = 2.5$, $\epsilon=1$}
		\label{fig:convergence_plot_rps_2_5}
\vspace{-0.3cm}
	\end{figure}
\vspace{-0.3cm}
		\begin{figure}[h]
\vspace{-0.3cm}
	\hfill
	\subfloat[Score trajectories]{\includegraphics[width=0.48\linewidth]{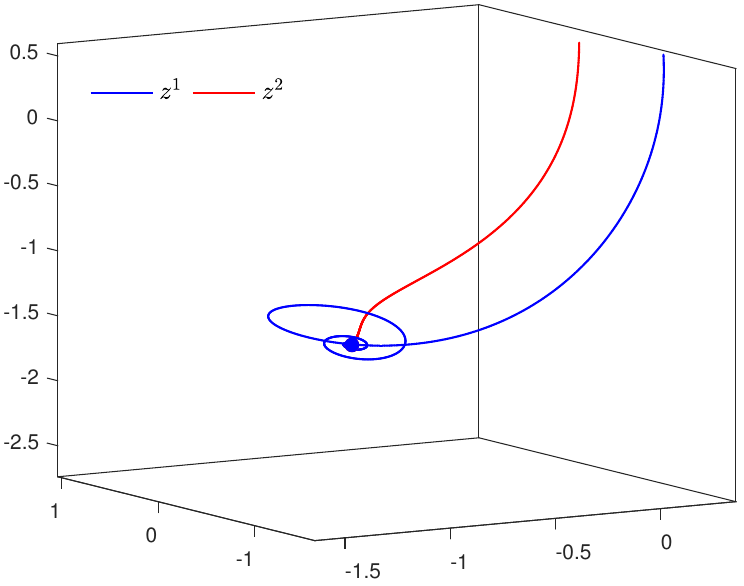}}
	\hfill
	\subfloat[Mixed strategy trajectories]{\includegraphics[width=0.5\linewidth]{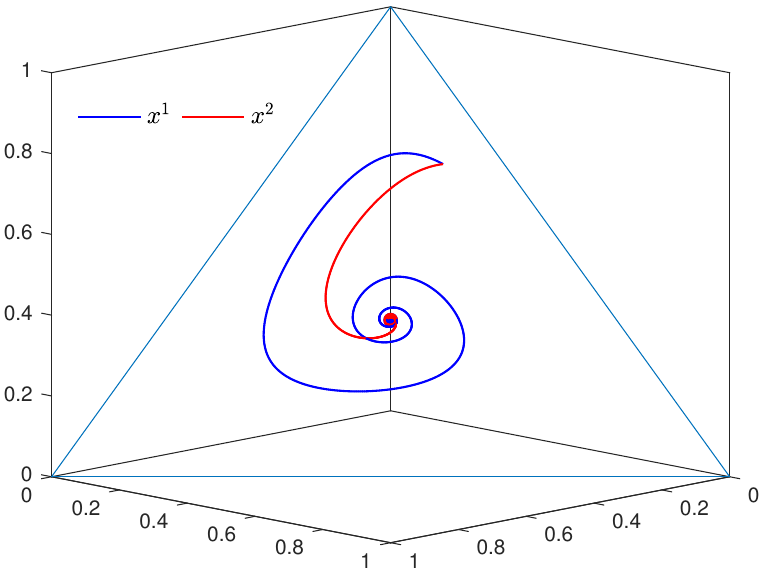}}
	\hfill
	\caption{Unstable RPS game, $l = 5$, $\epsilon=1$}
	\label{fig:convergence_plot_rps_5}
\vspace{-0.3cm}
\end{figure}
\vspace{-0.3cm}
		\begin{figure}[h]
\vspace{-0.3cm}
		\hfill
		\subfloat[Score trajectories]{\includegraphics[width=0.48\linewidth]{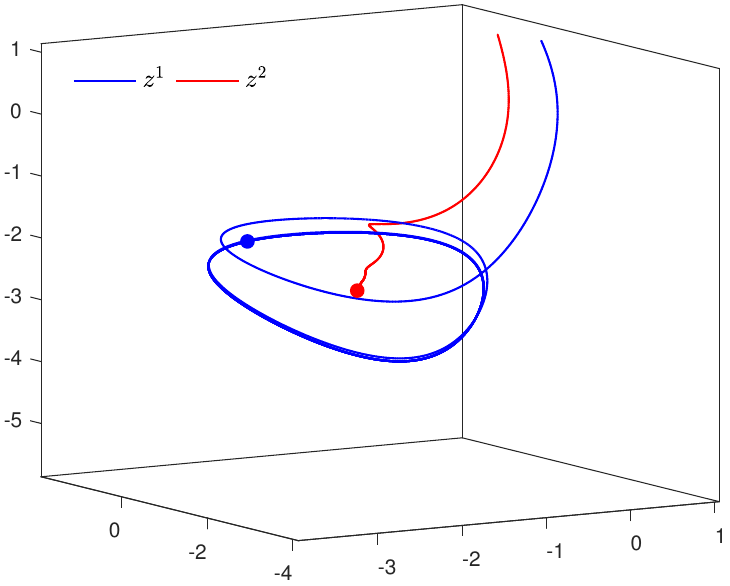}}
		\hfill
		\subfloat[Mixed strategy trajectories]{\includegraphics[width=0.5\linewidth]{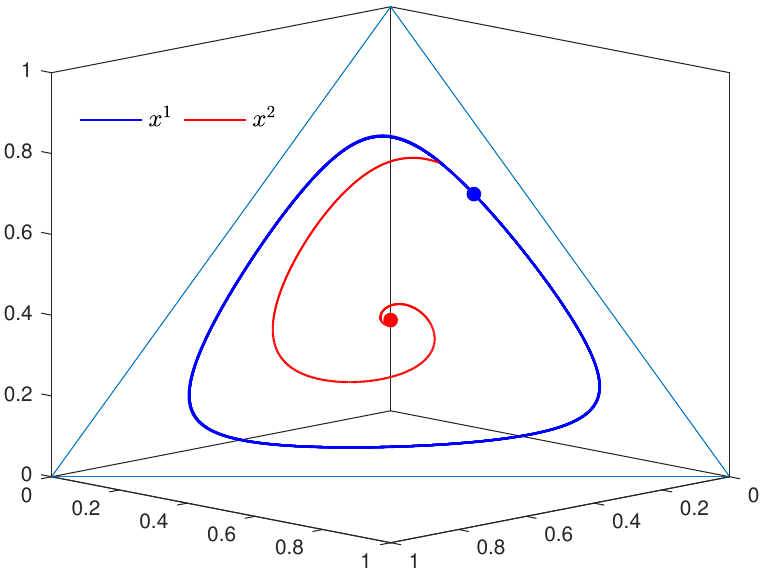}}
		\hfill
		\caption{Unstable RPS game, $l = 8$, $\epsilon=1$}
		\label{fig:convergence_plot_rps_8}
\vspace{-0.3cm}
	\end{figure}
\end{example}

\begin{example} $123$-Anti-Coordination Game 	


	Consider a population game in the class of concave potential games where the payoff vector $U(x)$ can be expressed as the gradient of a $C^1$, concave potential function, as in \cite{Sandholm}, \cite{Cominetti}. It is well-known that this class of potential games satisfies the monotonicity property in \autoref{assump:payoff_monotonicity}(i). Therefore  by \autoref{thm:stable_1} and \autoref{prop:exponentially_stable_prop}, the induced mixed strategy trajectories \eqref{eqn:first_order_score_dyn} and \eqref{eqn:closed_loop_feedback_system_2_p_2nd} are guaranteed to converge. 	Consider a well-known case  of the 123 anti-coordination game \cite{Sandholm}, with\vspace{-0.2cm} \begin{equation} \label{eqn:ant-123}	\mathbb{A} = \begin{bmatrix*}[r] -1 & 0 & 0 \\ 0 & -2 & 0 \\ 0 & 0 & -3 \end{bmatrix*}. 
\end{equation}The unique interior Nash equilibrium is located at $x^\star = [\frac{6}{11}; \frac{3}{11}; \frac{2}{11}] $. The Nash distribution (logit equilibrium) corresponding to $\epsilon = 1$ is $\overline{x}^\star= [ 0.40; 0.32;0.27] $ and it does not coincide with the Nash equilibrium. The simulation results are shown in \autoref{fig:anti-coord}(a), with the Nash equilibrium $x^\star $ shown as $\star$ in red. The induced mixed strategy trajectories of both \eqref{eqn:first_order_score_dyn} (shown in blue) and \eqref{eqn:closed_loop_feedback_system_2_p_2nd} (in red) converge to $\overline{x}^\star$. For small $\epsilon$,  $\overline{x}^\star$ gets arbitrarily close to $x^\star$ (e.g., \autoref{fig:anti-coord}(b) for $\epsilon = 0.1$).
		
\vspace{-0.3cm}
			\begin{figure}[h]
\vspace{-0.3cm}
		\hfill
		\subfloat[$\epsilon = 1$]{\includegraphics[width=0.5\linewidth]{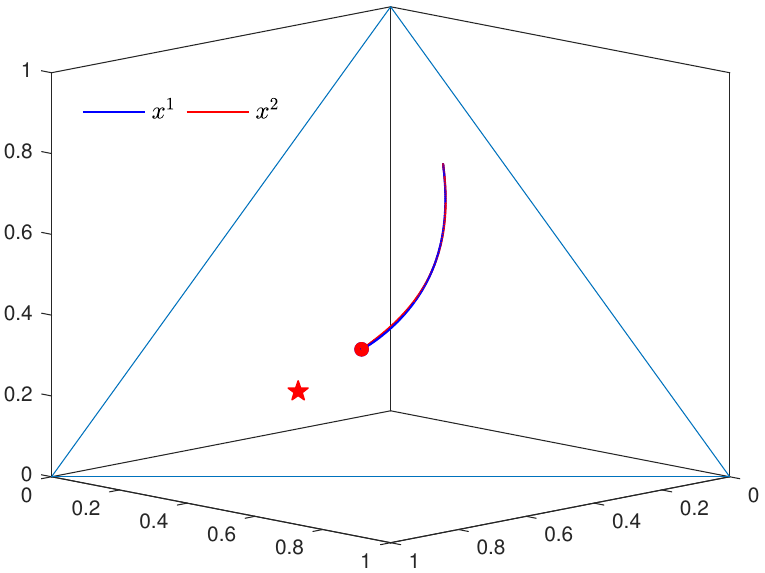}}
		\hfill
		\subfloat[$\epsilon = 0.1$]{\includegraphics[width=0.5\linewidth]{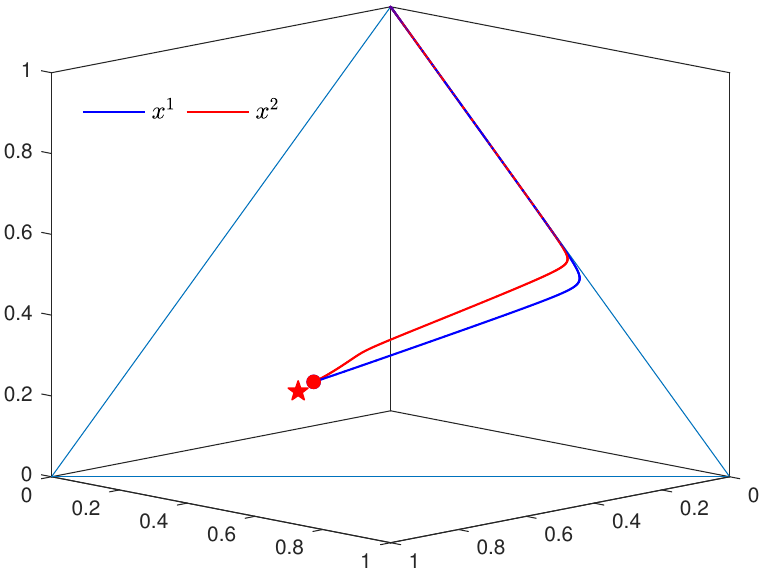}}
		\hfill
		\caption{123 Anti-Coordination Game}
		\label{fig:anti-coord}
\vspace{-0.3cm}
	\end{figure}

\end{example} \vspace{-0.2cm}

In the following three examples we consider two-player (asymmetric) games. We use the convention in \cite{Leslie},  \cite{Hofbauer}, where player 1 is the ``row" player with payoff matrix  $\mathbb{A}$ and player 2 is the ``column" player with payoff matrix  $\mathbb{B}$, so that  $\mathcal{U}^1(x) = (x^1)^\top  \mathbb{A} x^2$, $\mathcal{U}^2(x) = (x^1)^\top  \mathbb{B} x^2$. Then, $\mathcal{U}^1(x) = (x^1)^\top U^1(x)$, $\mathcal{U}^2(x) = (x^2)^\top  U^2(x)$  where $U^1(x) = \mathbb{A} x^2$, $U^2(x) = \mathbb{B}^\top  x^1$. With $x=(x^1,x^2)$,  
in this case,  \vspace{-0.2cm}\[	U(x) = \begin{bmatrix} U^1(x) \\ U^2(x) \end{bmatrix}= \begin{bmatrix} 0 & \mathbb{A} \\ \mathbb{B}^\top  & 0 \end{bmatrix} \begin{bmatrix}  x^1 \\  x^2 \end{bmatrix} :=\Phi x, \]
 Since  $U$ is linear, \eqref{eqn:U_Monotone_0} can be checked based on $\Phi$, hence on the payoff matrices of the two players, $\mathbb{A}$ and $\mathbb{B}$. Note that  \eqref{eqn:U_Monotone_0} is equivalent to $-(x-x^\prime)^\top  \Phi (x-x^\prime) \geq0$, for all $x, x^\prime \in \Updelta$, 
 hence if and only if $y^\top  (\Phi  +\Phi^\top ) y\leq 0$ for all $y \in  T\Updelta$. Thus \autoref{assump:payoff_monotonicity}(i) holds if $\Phi + \Phi^\top $ is negative semidefinite with respect to $T \Updelta$. Similarly,  \autoref{assump:payoff_monotonicity}(ii) holds for $\mu = \frac{1}{2} \lambda_{\max} (\Phi + \Phi^\top )$, where $\lambda_{\max} $ is maximum eigenvalue with respect to the tangent space $T \Updelta$. 
  \autoref{assump:payoff_monotonicity}(i) is met for  example in zero-sum games, where $\mathbb{B}= - \mathbb{A}$, since $\Phi = -\Phi^\top $ (skew-symmetric) and $x^\top  \Phi x =0$, for all $x \in \Updelta$. 

Note that $\Phi + \Phi^\top  =  \begin{bmatrix} 0 & \mathbb{A} + \mathbb{B} \\  \mathbb{A}^\top  +  \mathbb{B}^\top  & 0 \end{bmatrix}$, so that  the eigenvalues of $\Phi + \Phi^\top $ are the square roots of the eigenvalues of $ (\mathbb{A} + \mathbb{B}) (\mathbb{A} +  \mathbb{B})^\top $ (cf.  Lemma 4.4 in \cite{HopkinsHasy}). This means that $\Phi + \Phi^\top $ has always positive and negative eigenvalues, hence  all two-player games that satisfy 
\autoref{assump:payoff_monotonicity}(i)  are null stable (null monotone).  


\begin{example} Two-player Matching Pennies (MP) Game
	\label{mp}

 In this game, player $1$ and $2$ flip a coin and reveal them simultaneously. If the coins of the two players both land on head (or tail), then player $1$ earns a payoff $+1$ from player $2$, otherwise, player $2$ earns a payoff  $+1$ from player $1$. The payoff matrices are given by,	
	\begin{equation}
	\label{eqn:mp_game}
	\mathbb{A} = \begin{bmatrix*}[r] +1 & -1 \\ -1 & +1 \end{bmatrix*}, \quad  
	\mathbb{B} = -\mathbb{A}.
	\end{equation} so this is a zero-sum game. 
	The unique interior Nash equilibrium is given by $x^\star  = \begin{bmatrix} \frac{1}{2}  & \frac{1}{2} \end{bmatrix}^\top $, and coincides with the logit equilibrium $\overline{x}^\star$.  
The payoff  vector  is given by,
	\begin{equation} 
		U(x) = \begin{bmatrix} U^1(x) \\ U^2(x) \end{bmatrix}= \begin{bmatrix} 0 & \mathbb{A} \\ \mathbb{B}^\top  & 0 \end{bmatrix} \begin{bmatrix}  x^1 \\  x^2 \end{bmatrix} = \begin{bmatrix} \mathbb{A} x^2 \\ -\mathbb{A}^\top  x^1 \end{bmatrix}: =\Phi x.
	\end{equation}
Since $\Phi+\Phi^T\!=\!0$, 	\autoref{assump:payoff_monotonicity}(i) is satisfied. By \autoref{thm:stable_1} and \autoref{prop:exponentially_stable_prop}, 
for $l=1$, both first, \eqref{eqn:first_order_score_dyn},  and higher-order, \eqref{eqn:closed_loop_feedback_system_2_p_2nd} learning dynamics, are guaranteed to converge for any $\epsilon >0$. 
Simulations results are shown in  \autoref{mp_1_1} and \autoref{mp_1_2} for $\epsilon=1$ and $\gamma=1$,  and $\gamma=4$, respectively. 
The dotted line represents the induced strategy trajectories of \eqref{eqn:first_order_score_dyn}, whereas the solid line represents the induced strategy trajectories of \eqref{eqn:closed_loop_feedback_system_2_p_2nd}.  It can be seen from the simulation that all the mixed strategy trajectories converge to the Nash equilibrium and that a higher learning rate $\gamma$ increases the convergence speed.  

\vspace{-0.3cm}
\begin{figure}[h]
\vspace{-0.3cm}
		\centering
	\includegraphics[width= 0.8\linewidth]{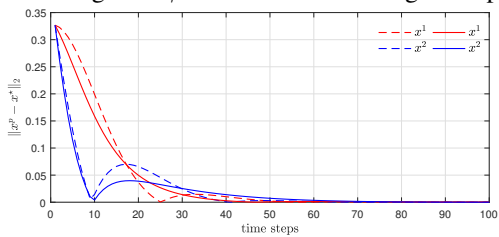} 
	\caption{Two-player Matching Pennies Game, $\gamma=1$}
	\label{mp_1_1}
\end{figure}
\vspace{-0.4cm}
\begin{figure}[h]
\vspace{-0.3cm}
		\centering
	\includegraphics[width= 0.8\linewidth]{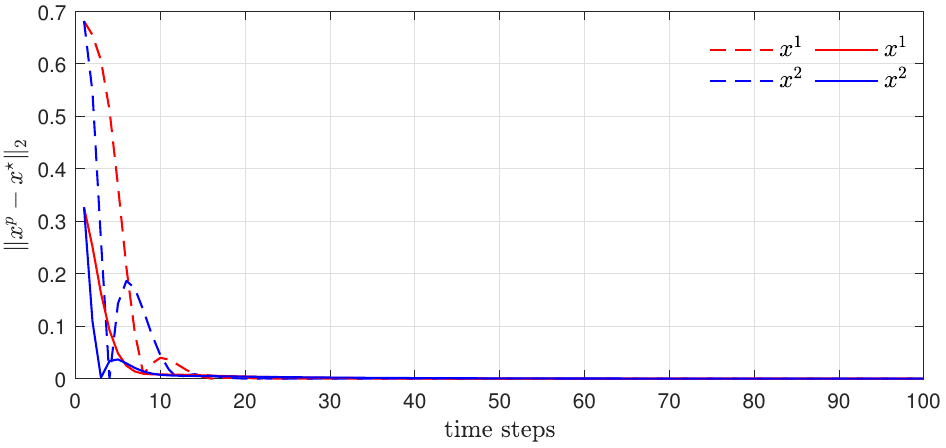}    
	\caption{Two-player Matching Pennies Game, $\gamma=4$}
	\label{mp_1_2}
\end{figure}

\end{example}

\begin{example} Two-player Rock-Paper-Scissors  Game \label{two_rps_game}
	
	Consider the two-player RPS game with payoff matrices 
$\mathbb{A}$  and $\mathbb{B}$, where $\mathbb{A}$ is as in \eqref{eqn:RPS_game} and $\mathbb{B}= \mathbb{A}^\top $.
 The unique Nash distribution (same as the Nash equilibrium $x^\star$) is $\overline{x}^\star \!=\!({\overline{x}^\star}^1,{\overline{x}^\star}^2)$  where  ${\overline{x}^\star}^1 \!=\!{\overline{x}^\star}^2 \! =\! (\frac{1}{3},  \frac{1}{3}, \frac{1}{3} ) $. 
 The payoff vector is\vspace{-0.1cm}
	\[	U(x) = \begin{bmatrix} U^1(x) \\ U^2(x) \end{bmatrix}= \begin{bmatrix} 0 & \mathbb{A} \\ \mathbb{B}^\top  & 0 \end{bmatrix} \begin{bmatrix}  x^1 \\  x^2 \end{bmatrix} = \begin{bmatrix} \mathbb{A} x^2 \\ \mathbb{A} x^1 \end{bmatrix}: =\Phi x,\] 
		where $\mathbb{A}$ is as in \eqref{eqn:RPS_game}. 
		The eigenvalues of $\mathbb{A}+ \mathbb{B} = \mathbb{A + \mathbb{A}^\top  } $ are  $\{ 2(l-1), 1-l, 1-l\}$ (see Example \ref{single_rps_game}). Then, the eigenvalues of $\Phi + \Phi^\top $ are   $\{ \pm 2(l-1),\pm(1-l), \pm (1-l) \}$.   For $l=1$, \autoref{assump:payoff_monotonicity}(i) holds; this is the zero-sum game considered in \cite{Leslie}. 		On the other hand, for $l\neq 1$ the game 
 satisfies  \autoref{assump:payoff_monotonicity}(ii) for $\mu = \frac{1}{2} |l-1|$. Thus, in this case even for $l <1 $ the game is unstable. By \autoref{thm:stable_1} and \autoref{prop:exponentially_stable_prop}, 
for $l=1$, both first, \eqref{eqn:first_order_score_dyn},  and higher-order, \eqref{eqn:closed_loop_feedback_system_2_p_2nd} learning dynamics, are guaranteed to converge for any $\epsilon >0$, while for $l \neq 1$, they are guaranteed to globally converge   for $\epsilon >  \frac{1}{2} |l-1|$. As in Example 	 \ref{single_rps_game}, the bound on $\epsilon$ is not tight. \autoref{fig:two_player_RPS_1_and_5} shows  the induced strategy trajectories of the two players under  \eqref{eqn:first_order_score_dyn} (dotted line) and  \eqref{eqn:closed_loop_feedback_system_2_p_2nd} (solid line), for $\epsilon =1$,  
	in the standard RPS game ($l = 1$) and unstable RPS game ($l=5$), respectively,  indicating convergence to $x^\star=\overline{x}^\star$. 

%
\vspace{-0.3cm}	
			\begin{figure}[h]
\vspace{-0.3cm}
		\hfill
		\subfloat[Two-player standard RPS, $l = 1$]{\includegraphics[width=0.5\linewidth]{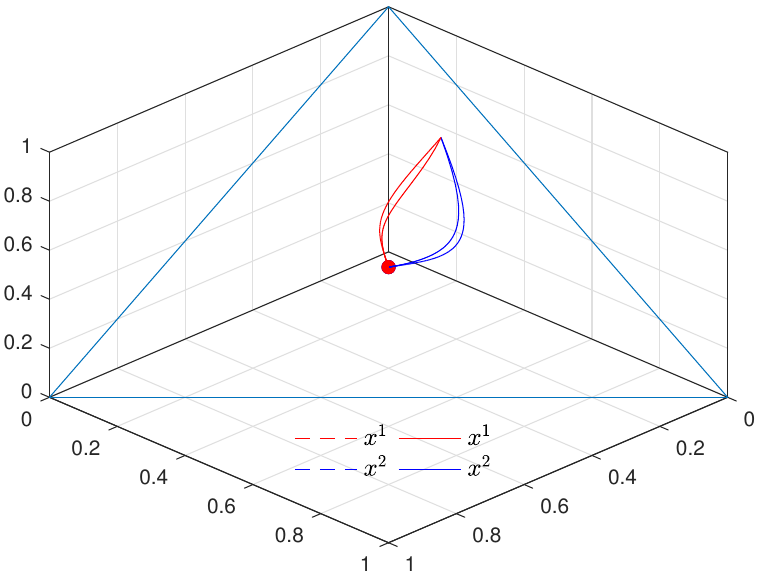}}
		\hfill
		\subfloat[Two-player unstable RPS, $l = 5$]{\includegraphics[width=0.5\linewidth]{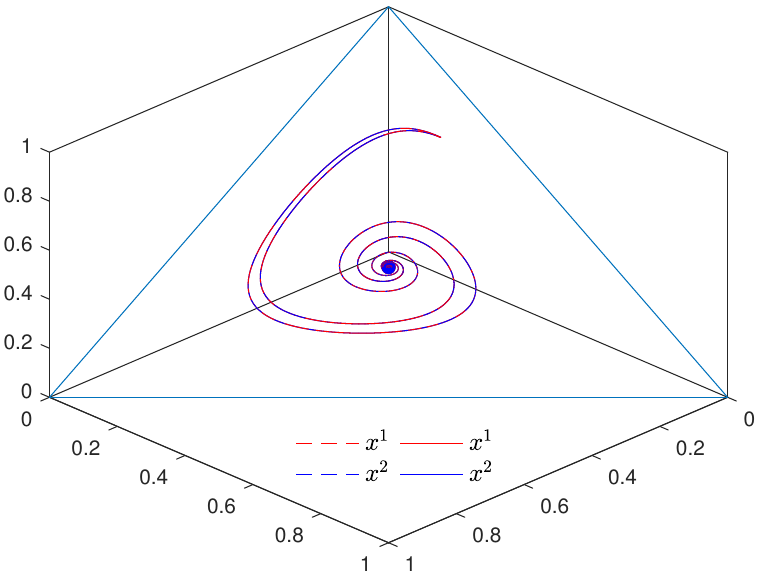}}
		\hfill
		\caption{Two-player standard and unstable RPS game, $\epsilon =1$}
		\label{fig:two_player_RPS_1_and_5}
\vspace{-0.3cm}
	\end{figure}
	We study the critical value of $\epsilon$ for which bifurcation occurs  by computing the eigenvalues of the Jacobian of  \eqref{eqn:first_order_score_dyn} and \eqref{eqn:closed_loop_feedback_system_2_p_2nd} at equilibrium. For $l=5$ the critical  $\epsilon$ value for bifurcation for the first-order dynamics  \eqref{eqn:first_order_score_dyn} is $\epsilon^*=\frac{2}{3}$, while numerically for \eqref{eqn:closed_loop_feedback_system_2_p_2nd} is  $\epsilon^*=0.347$, with stable dynamics for $\epsilon > \epsilon^*$.  In \autoref{two_player_5_lines_0_5} we show the error plots from equilibrium for  $\epsilon =0.5$ for both learning dynamics.  
	It can be seen that the higher-order dynamics \eqref{eqn:closed_loop_feedback_system_2_p_2nd}  converges, while \eqref{eqn:first_order_score_dyn} cycles, confirming the conclusion that for the same $\epsilon$ the higher-order dynamics can converge in a larger class of games (more hypo-monotone).
	\begin{figure}[htp!]
\vspace{-0.3cm}
		\centering
	\includegraphics[width= 0.8\linewidth]{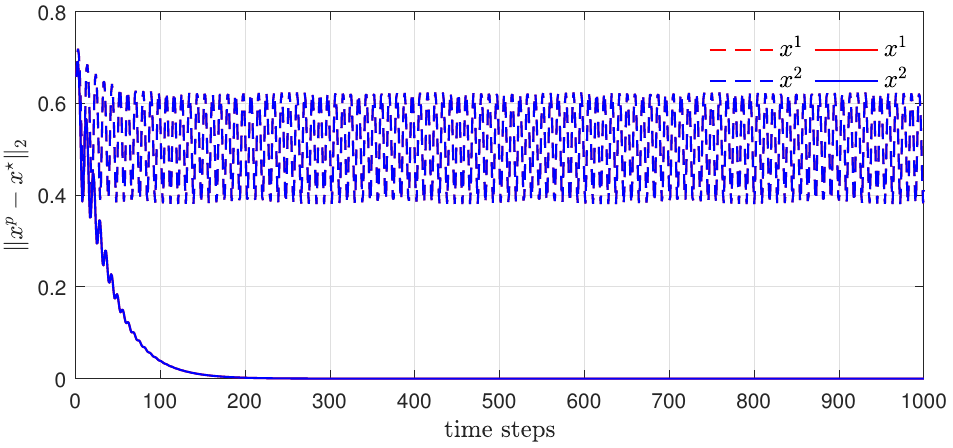} 	
	\caption{Two-player unstable RPS game, $l = 5$, $\epsilon=0.5$}
	\label{two_player_5_lines_0_5}
\vspace{-0.3cm}
	\end{figure}
		
\end{example}

\begin{example} Two-player Shapley Game

Consider the Two-player Shapley game discussed in \cite{Leslie}. The payoff matrices of this game are given by, 	
	\begin{equation}
		\mathbb{A} = \begin{bmatrix*}[r] 0 & 1 &  0 \\ 0 & 0  & 1 \\ 1 &  0 &  0  \end{bmatrix*} \quad \mathbb{B} = \mathbb{A}^\top .
	\end{equation}
and this is as $l=0$ in the RPS game (Example \ref{two_rps_game}), so $\mu =\frac{1}{2}$, an unstable game. 
In Fig.  \ref{fig:convergence_plot_rps_vs_shapley}(a) and (b) we compare the strategy trajectories  for player 1 under \eqref{eqn:first_order_score_dyn}  with  $\epsilon =0.1$, in the standard RPS game ($l=1$) (a) and in  the Shapley game ($l=0$) (b), respectively. We note that  in  the Shapley game a cycle is formed corresponding to Shapley triangle as in  \cite{Leslie}. 
\vspace{-0.1cm}
 		\begin{figure}[h]
\vspace{-0.3cm}
		\hfill
		\subfloat[Two-player standard RPS]{\includegraphics[width=0.5\linewidth]{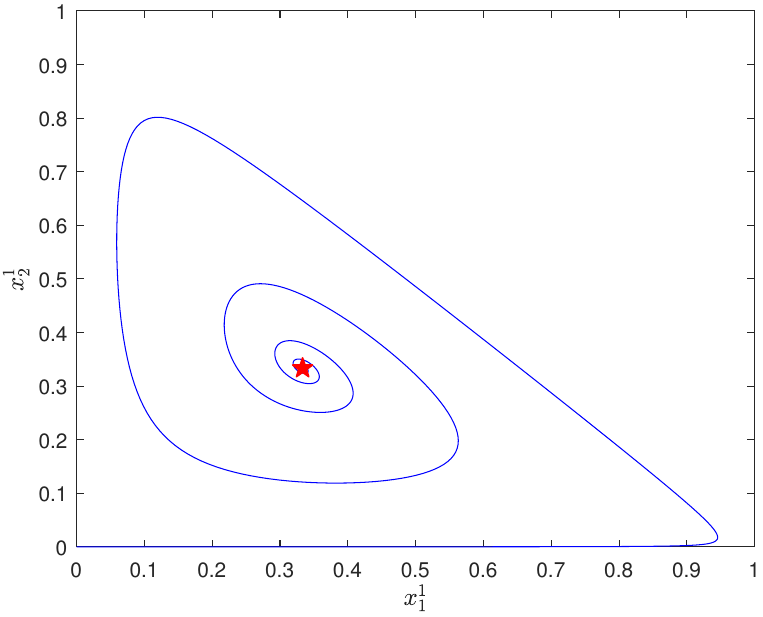}}  
		\hfill
		\subfloat[Shapley game]{\includegraphics[width=0.5\linewidth]{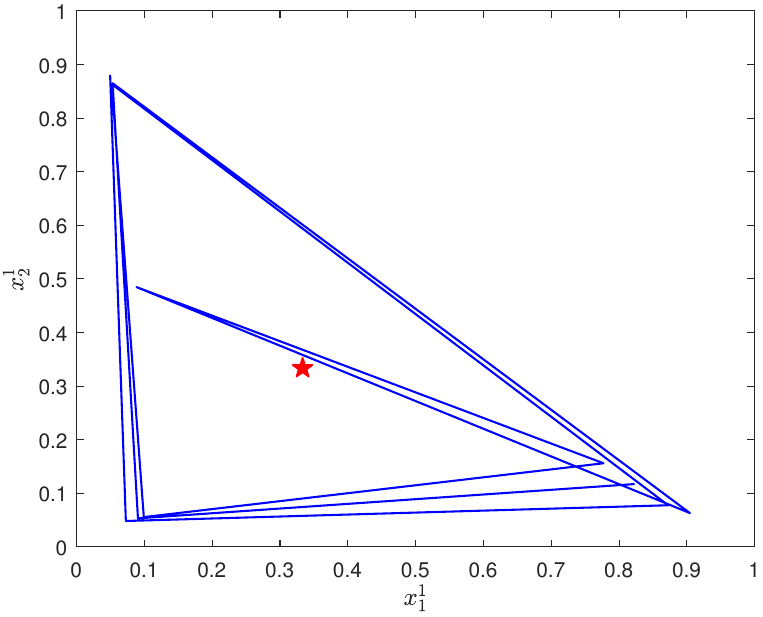}}  
		\hfill
		\caption{Two-player standard RPS and Shapley game, $\epsilon =0.1$}
		\label{fig:convergence_plot_rps_vs_shapley}
\vspace{-0.3cm}
	\end{figure}
\end{example}
However, by \autoref{thm:stable_1} and \autoref{prop:exponentially_stable_prop},   both learning dynamics  are guaranteed to converge for $\epsilon > \frac{1}{2}$. In \autoref{shapley} we show strategy trajectories under  \eqref{eqn:first_order_score_dyn} and \eqref{eqn:closed_loop_feedback_system_2_p_2nd}   for $\epsilon =1$; it can be seen that both \eqref{eqn:first_order_score_dyn} (dotted line) and  \eqref{eqn:closed_loop_feedback_system_2_p_2nd} (solid line) converge. Unlike \cite{Leslie}, there is no need for player-dependent learning rates that satisfy a singular perturbation specific relationship, (cf. PDRL in Proposition 5.4, \cite{Leslie}).
\vspace{-0.3cm}
	\begin{figure}[htp]
		\centering
				\includegraphics[width= 0.5\linewidth]{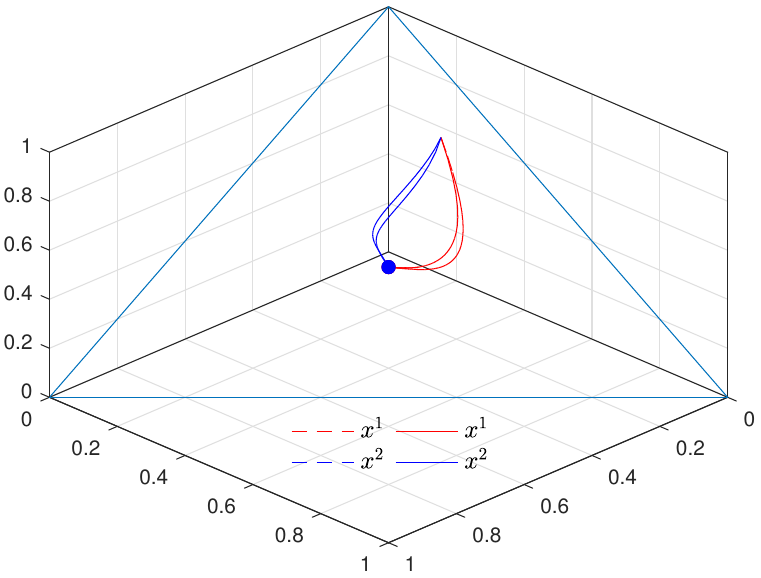} 	
		\caption{Two-player Shapley game, $\epsilon =1$}
		\label{shapley}
	\end{figure}

\begin{example} Three-Player Network Zero-Sum Game
	
Consider a network represented by a finite, fully connected, undirected graph $\mathbb{G} = (\mathcal{N}, \mathcal{E})$ where $\mathcal{N}$ is the set of vertices  (players) and $\mathcal{E} \subset \mathcal{N} \times \mathcal{N}$ is the set of edges. 
Given two vertices (players) 
$p,q \in \mathcal{N}$, 
we assume that there is a zero-sum game on the edge 
$(p,q)$ given by the payoff matrices $(\mathbb{A}^{p,q}, \mathbb{A}^{q,p})$, whereby $\mathbb{A}^{q,p} = -{\mathbb{A}^{p,q}}$. Assume that $N=3$ players are arranged in network $\mathbb{G}$ as shown in \autoref{three_player_graph} and the payoff matrix is that of a Matching Pennies (MP) game, \vspace{-0.2cm}\begin{equation}
\label{eqn:network_mp_game}
\mathbb{A}(k) = \begin{bmatrix*}[r] +k & -k \\ -k & +k \end{bmatrix*}.
\end{equation}
with the Nash equilibrium at 
$ \begin{bmatrix} \frac{1}{2}  & \frac{1}{2} \end{bmatrix}^\top $. 
	\begin{center}
		\begin{figure}[htp!]
\vspace{-0.2cm}
				\centering
			\includegraphics[width = 0.25\linewidth]{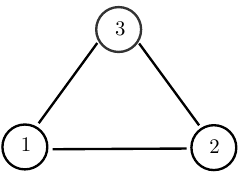}
			\caption{Three-Player Network Zero-Sum Game}
			\label{three_player_graph}
\vspace{-0.4cm}
		\end{figure} 
	\end{center}

\vspace{-0.4cm}
 Let the payoff matrices for each edge be given as,\vspace{-0.2cm}
\[ \begin{array}{lll}%
\mathbb{A}^{1,2} = \mathbb{A}(1) & \mathbb{A}^{1,3} = \mathbb{A}(2) & \mathbb{A}^{2,3} = \mathbb{A}(3)\\   
\mathbb{A}^{2,1} = -{\mathbb{A}^{1,2}}  &  \mathbb{A}^{3,1} = -{\mathbb{A}^{1,3}}   &  \mathbb{A}^{3,2} = -{\mathbb{A}^{2,3}}
\end{array}\]
Since each pair-wise interaction between players is a zero-sum game, the payoff vector of the overall player set is given by \vspace{-0.2cm}
$$U(x) =  \begin{bmatrix} U^1(x) \\ U^2(x)\\U^3(x) \end{bmatrix} =  \begin{bmatrix} 0 &  \mathbb{A}^{1,2} & \mathbb{A}^{1,3}\\ -{\mathbb{A}^{1,2}}^\top  & 0 & \mathbb{A}^{2,3} \\ -{\mathbb{A}^{1,3}}^\top  & -{\mathbb{A}^{2,3}}^\top  & 0 \end{bmatrix}\begin{bmatrix}\vphantom{-{\mathbb{A}^{1,2}}^\top } x^1 \\ x^2 \\ x^3  \vphantom{-{\mathbb{A}^{1,2}}^\top }\end{bmatrix}.
$$
Thus $U(x) = \Phi x$, where $\Phi+\Phi^T=0$, so that this game is a null monotone game, cf. \autoref{assump:payoff_monotonicity}(i). Hence, by \autoref{thm:stable_1} and \autoref{prop:exponentially_stable_prop}, \eqref{eqn:first_order_score_dyn} and \eqref{eqn:closed_loop_feedback_system_2_p_2nd}, respectively, converge for any $\epsilon >0$. Strategy trajectories of \eqref{eqn:first_order_score_dyn} (dotted line) and \eqref{eqn:closed_loop_feedback_system_2_p_2nd} (solid line) are plotted in \autoref{network_mp}, indicating 
convergence. 
\vspace{-0.1cm}
	\begin{figure}[htp!]
\vspace{-0.3cm}
		\centering
		\includegraphics[width= 0.8\linewidth]{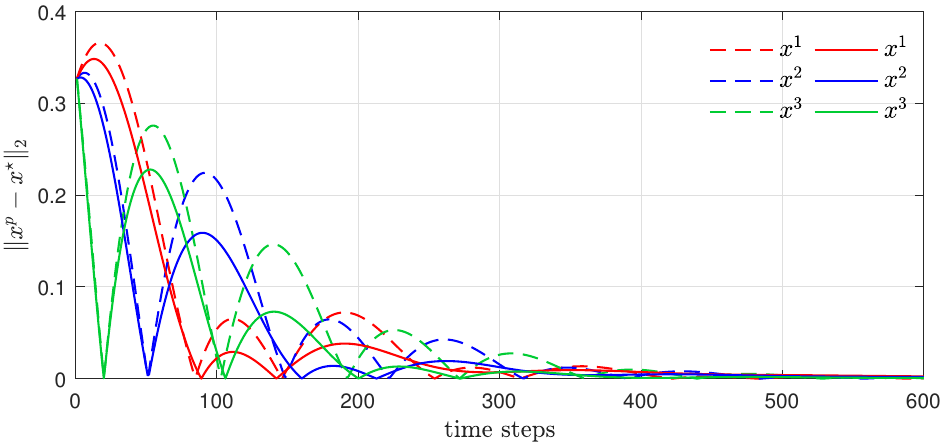}
		\caption{Three-Player Network MP game}
		\label{network_mp}
\vspace{-0.4cm}
	\end{figure}
		
\end{example}

\begin{example} Jordan Three-Player Matching Pennies Game
	
 The three-player Jordan game is a generalization of the Two-Player Matching Pennies game \cite[p. 22]{Fudenberg}. In this game, each player flips a coin and reveals them simultaneously. Player $1$ wins if the outcome matches that of player $2$. Player $2$ wins if the outcome match that of player $3$. Player $3$ wins if the outcome does not match that of player $1$. Let player $1$ be the row player, and player $2$ be the column one. 
 The payoff matrices can be represented as in  \autoref{table:jordan_payoff}.	 
	\begin{table}[h]
\vspace{-0.2cm}
		\centering
		\setlength{\extrarowheight}{1.5pt}
		\begin{tabular}{*{4}{c|}}
			\multicolumn{2}{c}{} & \multicolumn{2}{c}{Player $3$ chooses H}\\\cline{3-4}
			\multicolumn{1}{c}{} &  & $H$  & $T$ \\\cline{2-4}
			\multirow{2}*{}  & $H$ & $+1,+1,-1$ & $-1,-1,-1$ \\\cline{2-4}
			& $T$ & $-1,+1,+1$ & $+1,-1,+1$ \\\cline{2-4}
		\end{tabular}
		\thinspace
		\begin{tabular}{*{4}{c|}}
			\multicolumn{2}{c}{} & \multicolumn{2}{c}{Player $3$ chooses T}\\\cline{3-4}
			\multicolumn{1}{c}{} &  & $H$  & $T$ \\\cline{2-4}
			\multirow{2}*{}  & $H$ & $+1,-1,+1$ & $-1,+1,+1$ \\\cline{2-4}
			& $T$ & $-1,-1,-1$ & $+1,+1,-1$ \\\cline{2-4}
		\end{tabular}
\vspace{-0.1cm}
		\caption{Payoff Matrices of Jordan MP game}
		\label{table:jordan_payoff}
		\vspace{-0.2cm}
	\end{table}
The unique Nash equilibrium (and the Nash distribution/logit equilibrium) is at $\begin{bmatrix} \frac{1}{2} & \frac{1}{2} \end{bmatrix}^\top $ for each player.  
Even though the payoff vector is no longer linear, it can be shown that it satisfies \autoref{assump:payoff_monotonicity}(i). 
The payoff vector for the player set can be written as\begin{equation*}U(x) =  \begin{bmatrix} U^1_1(x) \\ U^1_2(x) \\ U^2_1(x) \\ U^2_2(x) \\  U^3_1(x) \\ U^3_2(x)\end{bmatrix} = \begin{bmatrix} \hphantom{-} x_1^2x_1^3 - x_2^2x_1^3 + x_1^2x_2^3 - x_2^2x_2^3\\ -x_1^2x_1^3 + x_2^2x_1^3 - x_1^2x_2^3+ x_2^2x_2^3\\ \hphantom{-}x_1^1x_1^3 + x_2^1x_1^3 - x_1^1x_2^3 - x_2^1x_2^3 \\ -x_1^1x_1^3 - x_2^1x_1^3 + x_1^1x_2^3 + x_2^1x_2^3\\  -x_1^1 x_1^2 + x_2^1 x_1^2 - x_1^1 x_2^2 + x_2^1 x_2^2 \\ \hphantom{-}x_1^1 x_1^2 - x_2^1 x_1^2 + x_1^1 x_2^2 - x_2^1 x_2^2\end{bmatrix}.\end{equation*}

Note that $\!U^1_1(x)\!= \!x_1^2(x_1^3 \!+ \!x_2^3) \!-\! x_2^2(x_1^3 \!+\!x_2^3) \!= \! x_1^2\!-\! x_2^2 \!=\! [1 \, \, -1 ] x^2$, 
since $x^3 \in \Updelta^3$, $x_1^3 +x_2^3=1$, and $ U^1_2(x)= -  U^1_1(x)$.
Similarly, $U^2_1(x) = (x_1^1 + x_2^1) x_1^3 -( x_1^1 + x_2^1)x_2^3 = x_1^3 -x_2^3 = [1 \, \, -1 ] x^3 $ and  $ U^2_2(x) = -  U^2_1(x)$. Also, $U^3_1(x)  = -x_1^1 (x_1^2+ x_2^2)  + x_2^1 (x_1^2 + x_2^2) =  -x_1^1 + x_2^1 =  [-1 \, \, 1 ] x^1$  and $U^3_2(x) = -U^3_1(x)$. 
Thus, with $x  = (x^1, x^2,  x^3)$, 
\begin{equation*}U(x)  
= \left [\begin{array}{cccccc}
 0 & 0 & 1 & -1 & 0 & 0 \\
 0 & 0 &-1 & 1 & 0 & 0 \\
 0 & 0 & 0 &  0 & 1 & -1 \\
 0 & 0 & 0 & 0 & -1 & 1 \\
 -1 & 1 & 0 & 0 & 0 & 0 \\
 1 & -1 & 0 & 0 & 0 & 0
\end{array} \right ] x :=  \Phi x.
\end{equation*}

\begin{figure}[htp!]
\vspace{-0.3cm}
	\centering
	\includegraphics[width= 0.8\linewidth]{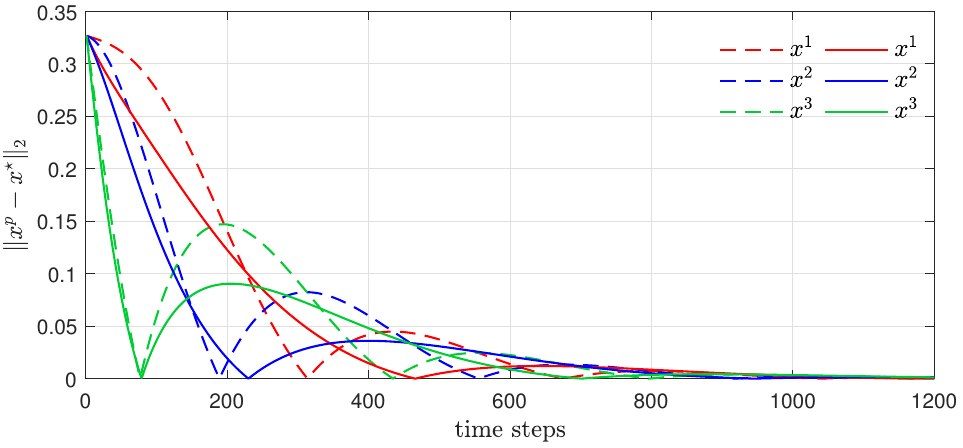}
	\caption{Three-player Jordan MP game, $\epsilon=1$}
	\label{jordan_mp}
\vspace{-0.2cm}
\end{figure}

The eigenvalues of $\Phi +  \Phi^\top$  are $\{ -4, 2,2,0,0,0 \}$. Thus $\mu = 1$ and the learning dynamics are guaranteed to converge for any $\epsilon >1$ (cf. \autoref{thm:stable_1} and \autoref{prop:exponentially_stable_prop}, under cf.  \autoref{assump:payoff_monotonicity}(ii)). Figure \autoref{jordan_mp}  shows strategy trajectories of all players under  \eqref{eqn:first_order_score_dyn} (dotted line) and \eqref{eqn:closed_loop_feedback_system_2_p_2nd} (solid line), indicating  convergence to the Nash equilibrium. 

\end{example}

In order to study the behaviour of  EXP-D-RL  \eqref{eqn:first_order_score_dyn}  and  H-D-EXP-RL\eqref{eqn:closed_loop_feedback_system_2_p_2nd} in some difficult games, in the following we consider two modified RPS games (as in \cite{BenaimHofbauerHopkins2009}), and a modified  asymmetric Jordan game (as in 
\cite{ShammaTAC2005}).

\begin{example} Modified RPS Game

Consider now two modified  (generalized) RPS games as  in \cite{BenaimHofbauerHopkins2009}). The first game has payoff matrix $\mathbb{A} = \begin{bmatrix*}[r] 0 & -1 & 3 \\ 2 & 0 & -1 \\ -1 & 3 & 0 \end{bmatrix*}$,  with unique Nash equilibrium at $x^\star = (0.40625, 0.3125, 0.28125)$, while the other game has payoff matrix  
$\overline{\mathbb{A}} = \begin{bmatrix*}[r] 0 & -3 & 1 \\ 1 & 0 & -2 \\ -3 & 1 & 0 \end{bmatrix*}$ (monocyclic), respectively, and the Nash equilibrium is at $x^\star = (0.28125, 0.3125, 0.40625)$. The eigenvalues of $\mathbb{A} + \mathbb{A} ^T$ are at $\{    3.3723, 
   -2.3723,   -1 \}$, while the eigenvalues of $\overline{\mathbb{A}}+ \overline{\mathbb{A}}^T$ 
are at  $\{    -3.3723,     2.3723,     1\}$, with the $-1$ and $+1$ eigenvalues corresponding to an eigenvector in the tangent space. Thus the first modified RPS game $\mathbb{A}$ is stable (monotone), while the second modified RPS game $\overline{\mathbb{A}}$ is unstable (hypo-monotone with $\mu =0.5$). By \autoref{thm:stable_1} and \autoref{prop:exponentially_stable_prop}, in the first RPS game $\mathbb{A}$, the two dynamics, \eqref{eqn:first_order_score_dyn}  and  \eqref{eqn:closed_loop_feedback_system_2_p_2nd}, are guaranteed to converge to a Nash distribution for any $\epsilon >0$, while in the second RPS game $\overline{\mathbb{A}}$, the two dynamics are guaranteed to converge to a Nash distribution for any $\epsilon >0.5$.  In Fig. \ref{fig:Benaim_RPS_games_eps1} and Fig. \ref{fig:Benaim_RPS_games_eps0_2} we show simulations in both games (side by side (a) and (b)), for $\epsilon =1$ and $\epsilon =0.2$, respectively, with trajectories for  \eqref{eqn:first_order_score_dyn} in blue  and those for  \eqref{eqn:closed_loop_feedback_system_2_p_2nd} in red, respectively. It can be seen that for $\epsilon =1$ in both games, both dynamics converge to the corresponding Nash distribution (logit equilibrium), which in either game does not coincide with the Nash equilibrium (marked with $\star$ in the plots).  For  $\epsilon = 1$, in the first RPS game $\mathbb{A}$, the Nash distribution is at $\overline{x}^\star=(0.379, 0.2997, 0.3213)$, while for the second RPS game $\overline{\mathbb{A}}$,  the Nash distribution is at $\overline{x}^\star=(0.2741, 0.3647, 0.3612)$. 
On the other hand for $\epsilon =0.2$,  in the first RPS game $\mathbb{A}$, the Nash distribution  is  $\overline{x}^\star=(0.4025, 0.3024, 0.2951)$ (very close to the Nash equilibrium $x^\star$), and as seen in Fig. \ref{fig:Benaim_RPS_games_eps0_2}(a),   both dynamics converge to it. This cannot be said for the second  RPS game $\overline{\mathbb{A}}$ for  $\epsilon =0.2$; as seen in  Fig. \ref{fig:Benaim_RPS_games_eps0_2}(b), the first-order dynamics \eqref{eqn:first_order_score_dyn} has a limit cycle, while the higher-order dynamics  \eqref{eqn:closed_loop_feedback_system_2_p_2nd} converges  to the corresponding Nash distribution, which is $\overline{x}^\star= (0.2653, 0.3237, 0.4109)$ (very close to the corresponding Nash equilibrium $x^\star$). This confirms our previous observations regarding the better performance of the higher-order dynamics  \eqref{eqn:closed_loop_feedback_system_2_p_2nd}. We note that  for  $\epsilon =0.1$, in the second RPS game  $\overline{\mathbb{A}}$, both dynamics have a limit cycle.

		\begin{figure}[h]
\vspace{-0.4cm}
		\hfill
		\subfloat[Modified RPS game $\mathbb{A}$ ]{\includegraphics[width=0.5\linewidth]{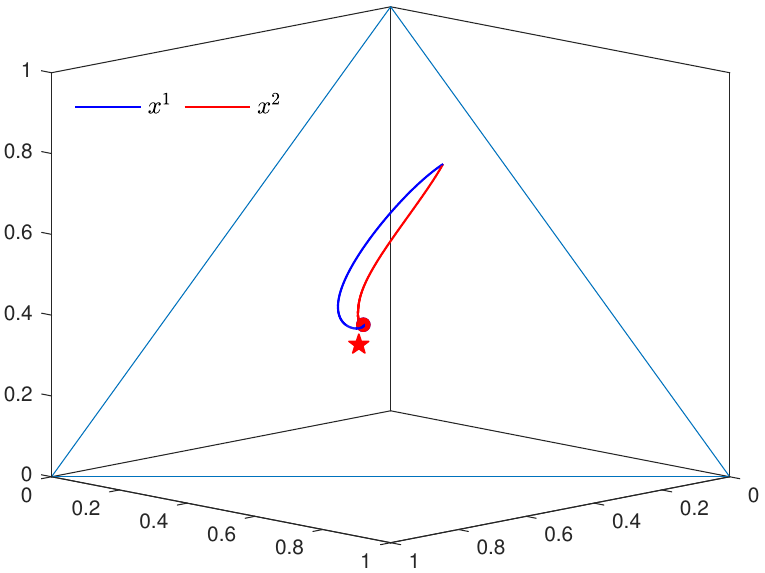}}  
				\hfill
		\subfloat[Modified RPS game $\overline{\mathbb{A}}$ ]{\includegraphics[width=0.5\linewidth]{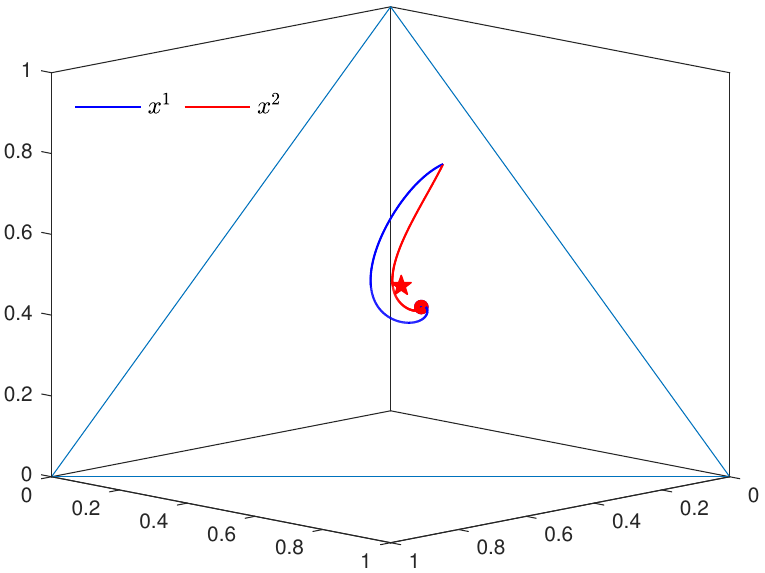}}  
		\hfill
		\caption{Modified RPS game $\mathbb{A}$  and $\overline{\mathbb{A}}$, \cite{BenaimHofbauerHopkins2009}, $\epsilon =1$}
		\label{fig:Benaim_RPS_games_eps1}
\vspace{-0.3cm}
	\end{figure}
\vspace{-0.1cm}
		\begin{figure}[h]
\vspace{-0.4cm}
		\hfill
		\subfloat[Modified RPS game $\mathbb{A}$ ]{\includegraphics[width=0.5\linewidth]{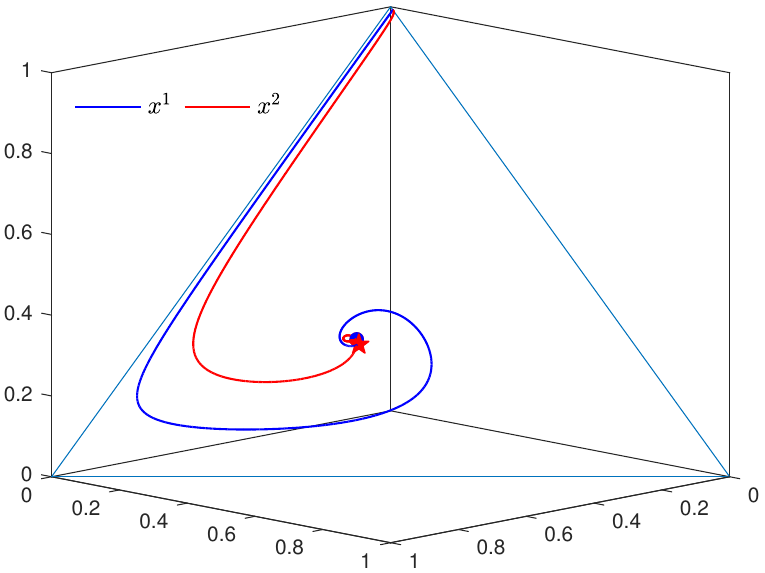}}  
		\hfill
		\subfloat[Modified RPS game $\overline{\mathbb{A}}$ ]{\includegraphics[width=0.5\linewidth]{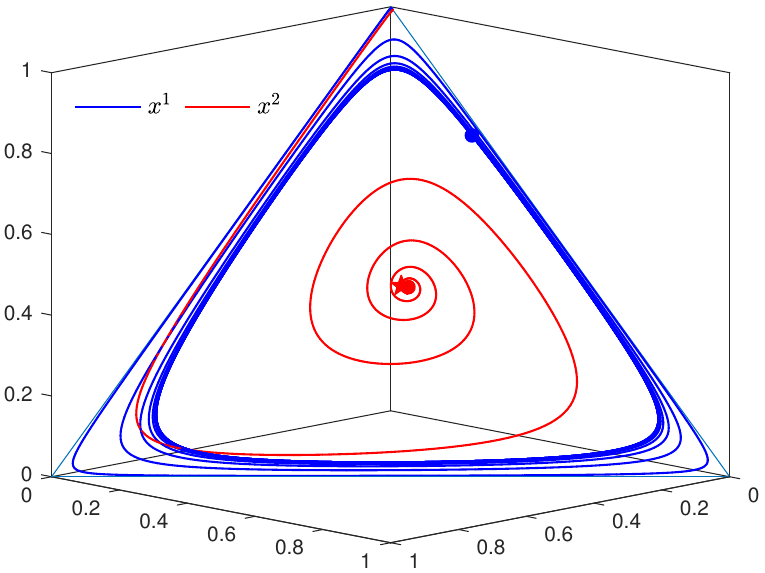}}  
		\hfill
		\caption{Modified RPS game $\mathbb{A}$  and $\overline{\mathbb{A}}$, \cite{BenaimHofbauerHopkins2009}, $\epsilon =0.2$}
		\label{fig:Benaim_RPS_games_eps0_2}
	\end{figure}

\end{example}

\begin{example} Modified Three-Player Jordan Game

Consider a three-player modified  (asymmetric) Jordan game as in 
\cite{ShammaTAC2005}, where  
$U^1(x) =  \left [\begin{array}{cc}
0 & 2  \\
1 & 0 
\end{array} \right ] x^2$, 
$U^2(x) =  \left [\begin{array}{cc}
0 & 1  \\
1 & 0 
\end{array} \right ] x^3$, and 
$U^3(x) =  \left [\begin{array}{cc}
0 & 1/3  \\
1 & 0 
\end{array} \right ] x^1$. The unique Nash equilibrium is at $(\frac{1}{4},\frac{3}{4}, \frac{2}{3},\frac{1}{3},\frac{1}{2},\frac{1}{2})$. In Fig. \ref{jordan_mp_modif} we display the strategy trajectory plots the three players for $\epsilon =0.1$, which show that the first-order dynamics \eqref{eqn:first_order_score_dyn} (dotted line) has a limit cycle, while the higher-order dynamics  \eqref{eqn:closed_loop_feedback_system_2_p_2nd} (solid line) converges to the Nash distribution, close to  Nash equilibrium. 

\begin{figure}[htp!]
\vspace{-0.3cm}
	\centering
	\includegraphics[width= 0.8\linewidth]{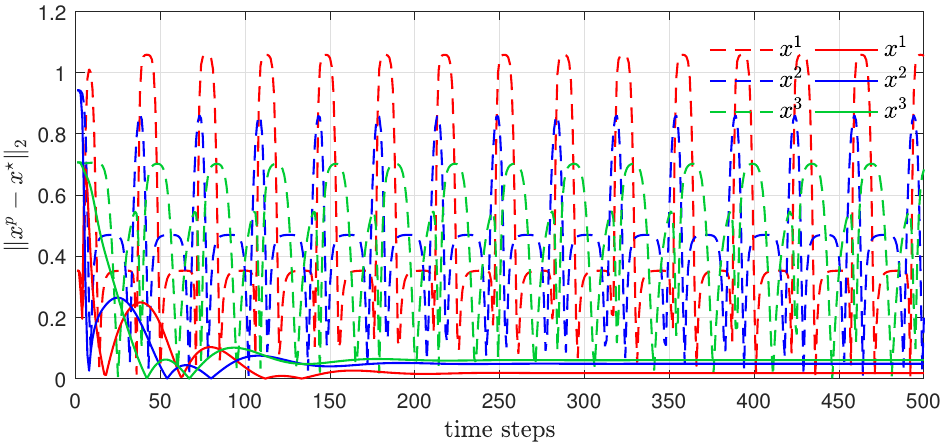}
	\caption{Modified Three-player Jordan game, $\epsilon=0.1$}
	\label{jordan_mp_modif}
\vspace{-0.4cm}
\end{figure}

\end{example}

\vspace{-0.2cm}
\section{Conclusion}

This paper presents a passivity-based approach for analyzing  first-order exponentially-discounted reinforcement learning  (EXP-D-RL) dynamics.
We have shown convergence for games characterized by their monotonicity property. We further exploited passivity properties to propose  a class of higher-order schemes that  preserve convergence properties, can improve  the speed of convergence and, as shown numerically, can even converge in cases whereby their first-order counterpart fail to converge. We demonstrated these properties through numerical simulations for several representative games.  


%

\begin{IEEEbiography}[{\includegraphics[width=1in]{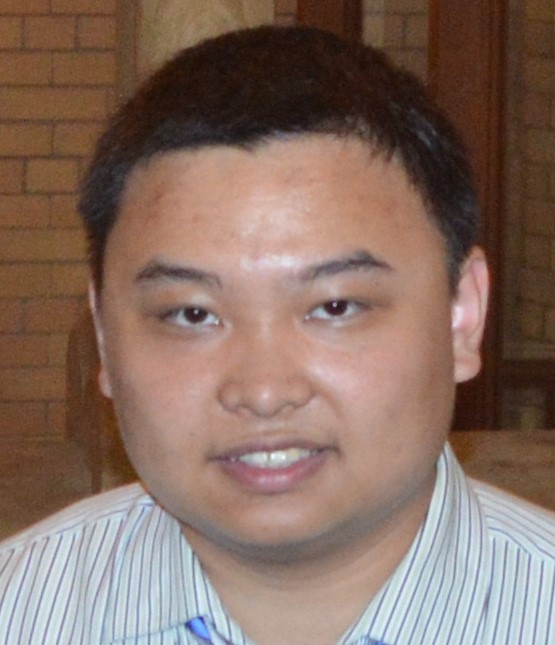}}]{Bolin Gao}
 is currently a Ph.D. student in the Systems Control Group at the University of Toronto. 
He received his B.A.Sc. (with Hons.) and M.A.Sc. degrees in Electrical Engineering from the University of Toronto, in 2015 and 2017, respectively. 
His research interests include design and control of game algorithms with applications to reinforcement learning in multi-agent systems. 
\end{IEEEbiography}

\begin{IEEEbiography}[{\includegraphics[width=0.9in]{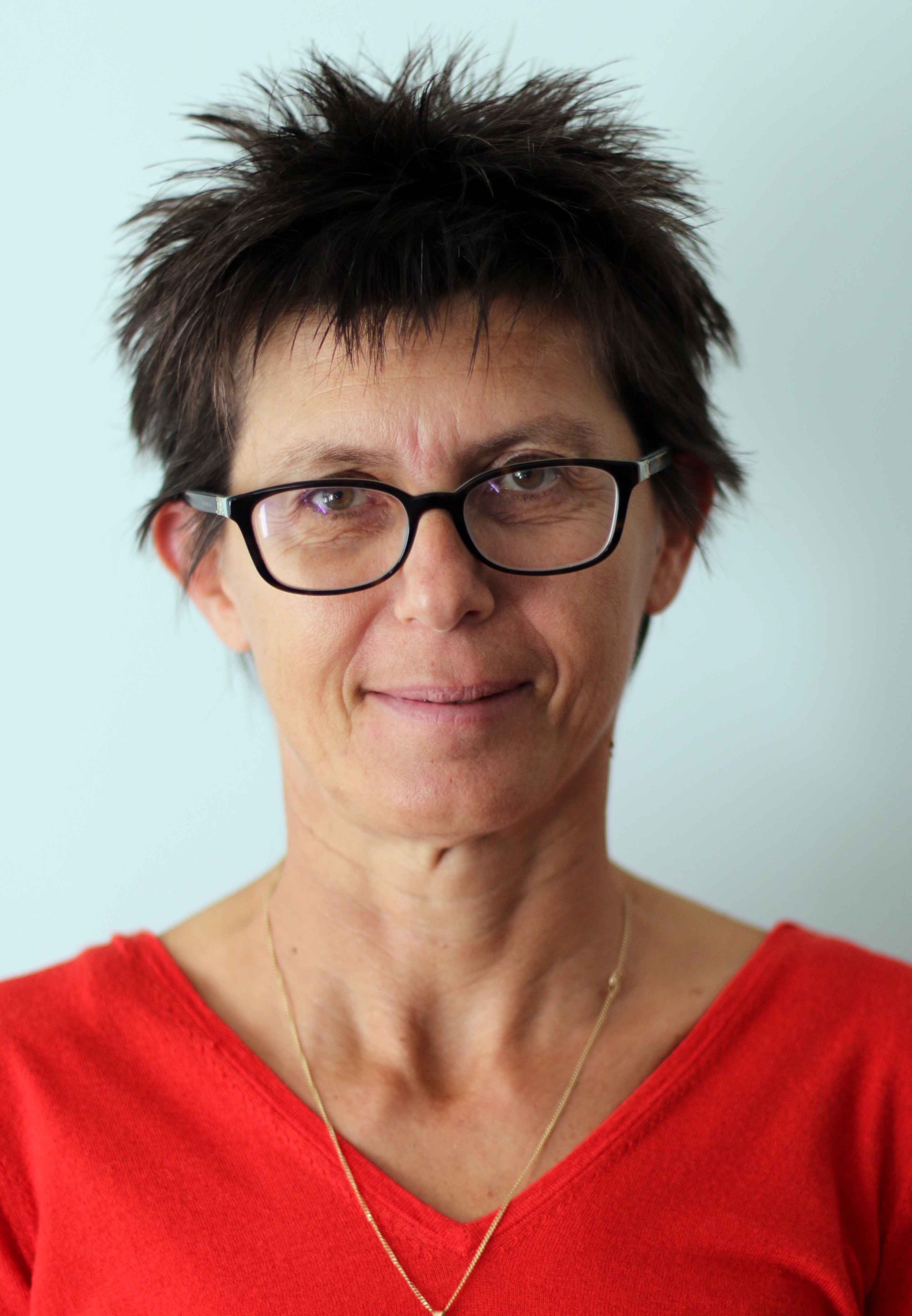}}]{Lacra Pavel}
(M'92 - SM'04)  received the Dipl. Engineer from Technical University of Iasi, Romania and the Ph.D. degree in Electrical Engineering from Queen's University, Canada.  After a postdoctoral stage at the National Research Council and four years of industry experience, in 2002 she joined University of Toronto, where she is now a Professor in the Systems Control Group, Department of Electrical and Computer Engineering. Her research interests are in game theory and distributed optimization in networks, with emphasis on dynamics and control.  She is the author of  the book {\em Game Theory for Control of Optical Networks} (Birkh\"{a}user-Springer Science, ISBN 978-0-8176-8321-4, 2012). She is an Associate Editor of \emph{IEEE Transactions on Control of Network Systems} and a Member of the Conference Editorial Board of the IEEE Control Systems Society; she acted as Publications Chair of the 45th IEEE Conference on Decision and Control.
\end{IEEEbiography}

\end{document}